\documentclass[ a4paper,10pt]{amsart}
\usepackage[foot]{amsaddr}

\usepackage[T1]{fontenc}

\usepackage[utf8]{inputenc}

\usepackage{lmodern}

\usepackage[english]{babel}

\usepackage{amsmath}
\makeatletter
\def\maketag@@@#1{\hbox{\m@th\normalfont\normalsize#1}}
\makeatother
\usepackage{amssymb}
\usepackage[alphabetic]{amsrefs}
\usepackage{mathtools}

\usepackage{mathrsfs}

\usepackage{xypic}

\usepackage{tikz-cd}

\usepackage{bookmark}

\usepackage{hyperref} 
\hypersetup{
	colorlinks=true,
	linkcolor=blue,
	filecolor=magenta,      
	urlcolor=cyan,
}

\usepackage{cleveref}

\usepackage{enumitem}
\usepackage{graphicx}
\usepackage{tikz}
\usepackage{tikz-cd}
\usetikzlibrary{matrix,arrows,decorations.pathmorphing}
\usepackage[all,cmtip]{xy}

\newtheorem*{thm-no-num}{Theorem}
\newtheorem*{df-no-num}{Definition}
\newtheorem{thm}{Theorem} [section]

\newtheorem{prop}[thm]{Proposition} 
\newtheorem{lm}[thm]{Lemma} 
\newtheorem{cor}[thm]{Corollary} 

\theoremstyle{remark}
\newtheorem{rmk}[thm]{Remark}
\newtheorem{ex}[thm]{Example}

\newcommand{\bA}{\mathbb{A}}
\newcommand{\Brm}{{\rm B}}
\newcommand{\PP}{\mathbb{P}}
\newcommand{\ZZ}{\mathbb{Z}}

\newcommand{\cl}[1]{\mathcal{#1}}

\newcommand*{\sheafhom}{\mathcal{H}\kern -.5pt om}

\newcommand{\Ccal}{\cl{C}}

\newcommand{\Mcal}{\cl{M}}
\newcommand{\Ocal}{\cl{O}}
\newcommand{\Scal}{\cl{S}}

\newcommand{\Hcal}{\cl{H}}
\newcommand{\Lcal}{\cl{L}}
\newcommand{\Xcal}{\cl{X}}
\newcommand{\Ycal}{\cl{Y}}

\newcommand{\Gcal}{\cl{G}}
\newcommand{\Pcal}{\cl{P}}

\newcommand{\Gm}{\mathbb{G}_{\rm m}}

\newcommand{\GLt}{\textnormal{GL}_3}
\newcommand{\PGLt}{\textnormal{PGL}_3}

\newcommand{\bfk}{\mathbf{k}}

\newcommand{\pr}{{\rm pr}}
\newcommand{\sm}{_{\rm sm}}

\renewcommand{\H}{{\rm H}}
\newcommand{\K}{{\rm K}}
\newcommand{\M}{{\rm M}}

\newcommand{\Inv}{{\rm Inv}^{\bullet}}
\newcommand{\Spec}{{\rm Spec}}

\newcommand{\irr}{\textnormal{irr}}
\newcommand{\binod}{\textnormal{bin}}
\newcommand{\fr}{\textnormal{fr}}

\begin{document}
	\title[Brauer groups of $\Mcal_3$, $\mathcal{A}_3$ and moduli of plane curves]{The Brauer groups of moduli of genus three curves, abelian threefolds and plane curves}
	\author[A. Di Lorenzo]{Andrea Di Lorenzo}
	\address{Università di Pisa, Dipartimento di matematica, Largo Bruno Pontecorvo 5, 56127 Pisa, Italy}
	\email{andrea.dilorenzo@unipi.it}
	\author[R. Pirisi]{Roberto Pirisi}
	\address{Università degli Studi di Napoli Federico II, Dipartimento di matematica e applicazioni R. Caccioppoli, Via Cintia, Monte S. Angelo I-80126, Napoli, Italy}
	\email{roberto.pirisi@unina.it}	
	\date{\today}
\begin{abstract}
We compute the $\ell$-primary torsion of the Brauer group of the moduli stack of smooth curves of genus three over any field of characteristic different from two and the Brauer group of the moduli stacks of smooth plane curves of degree $d$ over any algebraically closed field of characteristic different from two, three and coprime to $d$. We achieve this result by computing the low degree cohomological invariants of these stacks. As a corollary we are additionally able to compute the $\ell$-primary torsion of the Brauer group of the moduli stack of principally polarized abelian varieties of dimension three over any field of characteristic different from two.
\end{abstract}
\maketitle

\section*{Introduction}

\subsection*{Brauer groups of moduli problems}
Moduli stacks of curves are among the most studied objects in algebraic geometry. The problem of computing invariants of $\Mcal_g$ has attracted many mathematicians in the last sixty years: for instance, the Picard group of $\Mcal_g$ has been the subject of a number of papers, such as \cites{MumPic, ArbCor87, FulOl, FriViv}, and the same can be said for its rational Chow ring (see for instance \cites{MumEnum, Fab, Iza, CL}).

The Brauer group is a widely studied invariant in number theory, dating back to work of Brauer and Noether on the Brauer group of a field. It was later generalized to algebraic varieties and schemes, and up to the vast generality of Topoi, in Grothendieck's works. Therefore, one can naturally ask what the Brauer group of $\Mcal_g$ is, that is in which ways we can functorially associate to any smooth curve  $C \to S$ of genus $g$ an element $\phi(C) \in {\rm Br}(S)$.

So far, we know the following:
\begin{itemize}
    \item[$g=1$:] Using classical methods, Antieau and Meier \cite{AntMeiEll} computed the Brauer group of $\Mcal_{1,1}$ over a variety of bases, most notably the integers. Following this Shin \cite{Shi} computed the group over all finite or algebraically closed fields. The authors then introduced new techniques to extend the computation to arbitrary fields \cite{DilPirPositive}.
    \item[$g=2$:] The authors computed (up to $p$-primary torsion) the Brauer group of the stack $\Hcal_g$ of hyperelliptic curves of genus $g$ and its compactifications over any field of characteristic different from $2$ \cites{DilPirBr, DilPirRS}, thus obtaining a presentation for the Brauer group of $\Mcal_2$.
    \item[$g\geq 4$:] When the base field is the complex numbers, one can consider the analytic Brauer group, which can be computed through topological methods and is related to the regular Brauer group; in particular, putting together works of Arbarello-Cornalba \cite{ArbCor87}, Korkmaz-Stipsicz \cite{KorSt03} and Schr\"oer \cite{Sch05}, we get that over the complex numbers the Brauer group of $\Mcal_g$ is zero for $g \geq 4$ (see for example \cite{FriPirBrauer}*{Theorem 4.1} for how to derive the result).
\end{itemize}
To the authors' knowledge all that is known for $g=3$, using \cite{KorSt03}*{Theorem 1.2} in the same way as for $g\geq 4$, is that over $\mathbb{C}$ the Brauer group of $\Mcal_3$ is contained in $\ZZ/2\ZZ$. This paper's main aim is to close this gap:
\begin{thm-no-num}
Let $\bfk$ be a field of characteristic different from $2$, and let $\Mcal_3$ be the moduli stack of smooth genus three curves over $\bfk$. 
\begin{itemize}
    \item If ${\rm char}(\bfk)=0$ the Brauer group of $\Mcal_3$ is isomorphic to ${\rm Br}(\bfk) \oplus \ZZ/2\ZZ$. 
    \item If ${\rm char}(\bfk)=p>2$ then ${\rm Br}(\Mcal_3)={\rm Br}(\bfk) \oplus \ZZ/2\ZZ \oplus B_p$ where $B_p$ is a $p$-primary torsion subgroup.
\end{itemize}
\end{thm-no-num}
In particular, our theorem shows that even over $\mathbb{C}$ the Brauer group of $\Mcal_3$ is not trivial, differently from what happens for $g\geq 4$. 

Interestingly, the $\ZZ/2\ZZ$ factor in the presentation above is related to the theory of theta-characteristics of genus three curves, which is in turn related to two famous enumerative results: the $28$ bitangents of a smooth quartic and the $27$ lines over a smooth cubic surface.

As a relatively straightforward corollary of our result, we show that the same statement holds for the moduli stack of three dimensional principally polarized abelian varieties.

\begin{thm-no-num}
Let $\bfk$ be a field of characteristic different from $2$, and let $\mathcal{A}_3$ be the moduli stack of three dimensional principally polarized abelian varieties over $\bfk$. 
\begin{itemize}
    \item If ${\rm char}(\bfk)=0$ the Brauer group of $\mathcal{A}_3$ is isomorphic to ${\rm Br}(\bfk) \oplus \ZZ/2\ZZ$. 
    \item If ${\rm char}(\bfk)=p>2$ then ${\rm Br}(\mathcal{A}_3)={\rm Br}(\bfk) \oplus \ZZ/2\ZZ \oplus B'_p$ where $B'_p$ is a $p$-primary torsion subgroup.
\end{itemize}
\end{thm-no-num}

Moreover, let $\Xcal_{d}$ be the moduli stack parametrizing smooth plane curves of degree $d$. As an additional application of our techniques we compute, up to $p$-primary torsion, the Brauer group of $\Xcal_d$ over an algebraically closed field of characteristic different from $2,3$ and coprime to $d$.

\begin{thm-no-num}
Let $\bfk$ be an algebraically closed field of characteristic different from $2,3$ and coprime to $d$, and let $\Xcal_d$ be the moduli stack of smooth plane curves of degree $d$ over $\bfk$. Then:
\begin{itemize}
    \item If ${\rm char}(\bfk)=0$ the Brauer group of $\Xcal_d$ is isomorphic to $\ZZ/\langle d,6 \rangle$. 
    \item If ${\rm char}(\bfk)=p>0$ then ${\rm Br}(\Xcal_d)$ is isomorphic to $\ZZ/\langle d,6\rangle \oplus B_{d,p}$ where $B_{d,p}$ is a $p$-primary torsion subgroup.
\end{itemize}
\end{thm-no-num}

\subsection*{Cohomological invariants}

Our main tool will be the theory of cohomological invariants for algebraic stacks, first introduced by the second author in \cite{PirAlgStack} and expanded to the generality we use here by the authors in \cite{DilPirBr}; we should note that an even more general definition is given in \cite{DilPirPositive} to consider $p$-torsion coefficients in characteristic $p$, but it will not be needed here. 

Classically, cohomological invariants resemble characteristic classes: as codified in \cite{GMS}, they are considered for an algebraic group $G/\bfk$ and consist of ways to functorially assign a class $\alpha(E) \in \M^{\bullet}(F)$ to each $G$-torsor $E$ over a field $F$, where ${\rm M}$ is a cycle module (the main examples of Rost's cycle modules \cite{Rost} are Milnor's $\K$-theory and twisted Galois cohomology $\H^i_{D}(F)=\H^i_{\textnormal{Gal}}(F,D(i))$). As $G$-torsors over $F$ correspond to $F$-valued points of the classifying stack $\Brm G$, one can think of cohomolgical invariants as invariants of the classifying stack $\Brm G$ rather than the group. 

The idea can be readily extended to general algebraic stacks \cite{DilPirBr}*{Definition 2.3} by defining a cohomological invariant as a functorial assignment of an element $\alpha(P) \in \M(\bfk(P))$ to each point of $\Xcal$ (up to isomorphism), subject to a \emph{continuity condition} that is automatically satisfied in the classical case. When $\Xcal=\Brm G$ this definition agrees with the classical one, while when $\Xcal=X$ is a scheme we retrieve another known invariant, unramified cohomology. The cohomological invariants of $\Xcal$ with coefficients in $\M$ form a graded group $\Inv(\Xcal,\M)$ and when $\Xcal$ is a smooth quotient stack we get the equality
\[
{\rm Inv}^2(\Xcal, {\rm H}_{\mu_\ell^{\vee}})={\rm Br}'(\Xcal)_{\ell}
\]
where ${\rm H}^i_{\mu_\ell^{\vee}}(F)={\rm H}^i_{\textnormal{Gal}}(F,\ZZ/\ell\ZZ(i-1))$ denotes $(-1)$-twisted Galois cohomology and ${\rm Br}'(\Xcal)_\ell$ is the $\ell$-torsion of the cohomological Brauer group of $\Xcal$. Our result on the Brauer group of $\Mcal_3$ is in fact a consequence of the computation of the low degree cohomological invariants of $\Mcal_3$:

\begin{thm-no-num}
Let $\bfk$ be a field, let $\ell$ be a positive integer coprime to ${\rm char}(\bfk)$ and let $\M$ be a $\ell$-torsion cycle module. Then 
\[
{\rm Inv}^{\leq 2}(\Mcal_3,\M) = \M^{\leq 2}(\bfk) \oplus \alpha_2\cdot \M^{0}(\bfk)_2
\]
where $\alpha_2 \in {\rm Inv}^2(\Xcal, \K_2)$ is obtained on the open subset $\Mcal_3 \smallsetminus \Hcal_3$ as a pullback of the second Galois-Stiefel-Whitney class $\alpha_2 \in {\rm Inv}^2({\rm BS}_{28},\K_2)$.
\end{thm-no-num}

Our computation of the Brauer group of the stack $\Xcal_d$ of smooth plane curves of degree $d$, on the other hand, depends on the computation of the low degree cohomological invariants of the stack $\Xcal_d^{\textnormal{fr}}$ of \emph{framed} smooth plane curves of degree $d$, which is a $\Gm$-gerbe over $\Xcal_d$.

\begin{thm-no-num}
Let $\bfk$ be an algebraically closed field of characteristic different from $2$ and coprime to $d$, and let $\ell$ be an even positive integer that's not divisible by ${\rm char}(\bfk)$. Finally let $i_d$ be equal to $1$ if $3$ divides $d$, and $0$ otherwise. Then
\[{\rm Inv}^{\leq 2}(\Xcal_d^{\textnormal{fr}},\K_\ell)=\ZZ/\ell\ZZ \oplus \ZZ/\langle\ell,i_d(d-1)^2 \rangle\left[1\right] \oplus \ZZ/\langle 2,d \rangle\left[2\right].\] 
\end{thm-no-num}

\subsection*{Notation}

Throughout the paper every scheme and stack will be assumed to be of finite type (or a localization thereof) over a field $\bfk$ of characteristic different from $2$. Moreover when working with the stacks of degree $d$ plane curves we will always assume that ${\rm char}(\bfk)$ does not divide $d$. 

With the letter $\ell$ we will always mean a positive integer coprime to the characteristic of $\bfk$. If $A$ is an abelian group, by $A_r$ we will always mean the $r$-torsion subgroup of $A$. If $A^{\bullet}$ is a graded group we write $A^{\bullet}\left[ s \right]$ for the graded group obtained by shifting $A^{\bullet}$ up in degree by $s$. We write $A^{\leq d}$ for the subgrooup of elements of degree at most $d$.

With the notation ${\rm H}^i(\Xcal,A)$ we always mean \'etale cohomology with coefficients in $A$, or lisse-\'etale if $\Xcal$ is not a Deligne-Mumford stack. If $R$ is a $\bfk$-algebra we will often write ${\rm H}^i(R,A)$ for ${\rm H}^i(\Spec(R),A)$. In particular for a field $F/\bfk$ we have ${\rm H}^i(F,A)={\rm H}^i_{\textnormal{\'et}}(\Spec(F),A)={\rm H}_{\textnormal{Gal}}^i(F,A)$.

\subsection*{Acknowledgements} We wish to thank the referee for a thorough review containing a number of very helpful suggestions.

\section{Preliminaries}
In this section, we recall the theory of (equivariant) Chow groups with coefficients in a cycle module (\Cref{sec:cycle modules} and \Cref{sec:equi chow}). We then give a brief overview of the theory of cohomological invariants (\Cref{sec:coh}) and explain its connection to Chow groups with coefficients.
\subsection{Cycle modules}\label{sec:cycle modules}

Let $F$ be a field, and let ${\rm T}^{\bullet}(F)=\oplus_{i} (F^*)^{\otimes i}$ be the tensor algebra over $F$. We will write $\lbrace a_1, \ldots, a_n \rbrace$ for the element $a_1 \otimes \ldots \otimes a_n \in {\rm T}^{\bullet}(F)$, so that the multiplication reads $\lbrace a_1, \ldots, a_r \rbrace \cdot \lbrace b_1, \ldots, b_d \rbrace = \lbrace a_1, \ldots, a_r, b_1, \ldots, b_d \rbrace$.

Milnor's K-theory ring of $F$ is the graded ring $\K^{\bullet}_{\M}(F)={\rm T}^{\bullet}(F)/I$, where $I$ is the bilateral ideal generated by elements $\lbrace a_1, \ldots, a_r \rbrace$ such that $a_1 + \ldots + a_r = 1$. Some properties of $\K_{\rm Mil}^{\bullet}(F)$ we'll need are:
\begin{itemize}
    \item $\K_{\rm Mil}^{\bullet}(F)$ is graded-commutative, i.e. \[\lbrace a_1, \ldots , a_r, b_1, \ldots, b_d \rbrace = (-1)^{rd} \lbrace b_1, \ldots, b_d, a_1, \ldots , a_r \rbrace.\]
    \item $\lbrace a_1, \ldots, a_n \rbrace$ is zero if $a_1 + \ldots + a_n = 0$.
    \item $\lbrace a, a \rbrace = \lbrace -1, a \rbrace.$
\end{itemize}
Crucially, Milnor's K-theory induces a functor from the category of finitely generated field extensions of $\bfk$, and it has a series of important functorial properties
\begin{enumerate}
    \item For any finitely generated extension $f:F \to F'$ a map 
    \[f^*:\K^{\bullet}_{\M}(F) \to \K^{\bullet}_{\M}(F'), \quad f^*\lbrace a \rbrace = \lbrace f(a) \rbrace.\]
    \item For any finite extension $f:F \to F'$ a map 
    \[f_*:\K^{\bullet}_{\M}(F') \to \K^{\bullet}_{\M}(F), \quad f_*\lbrace a \rbrace = \lbrace {\rm N}^{F'}_F(a) \rbrace.\]
    \item For a DVR $(R,v)$ with quotient field $\bfk(R)$ and residue field $\bfk_v$ a ramification map
    \[\partial_v:\K^{\bullet}_{\M}(\bfk(R)) \to \K^{\bullet}_{\M}(\bfk_v), \quad \partial_v(\lbrace \pi, a_1, \ldots, a_r \rbrace)=\lbrace \overline{a_1}, \ldots, \overline{a_r} \rbrace \]
    where $\pi$ is a uniformizer for $R$, which lowers degree by one.
\end{enumerate}
The maps above satisfy the expected functorial compatibilities and a projection formula $f_*(f^*(\lbrace a_1, \ldots a_r\rbrace))={\rm deg}(F'/F)\lbrace a_1, \ldots a_r\rbrace$. Moreover we have the following exact sequence for Milnor's K-theory of a purely transcendental extension
\[0 \to \K_{\rm Mil}^{\bullet}(F) \to \K_{\rm Mil}^{\bullet}(F(t)) \xrightarrow{\partial} \bigoplus_{x \in \bA^1_F} \K_{\rm Mil}^{\bullet}(\bfk(x)) \to 0 \]
where the map $\partial$ is given by the sum of all ramification maps coming from points of codimension one in $\bA^1_F$.

A \emph{cycle module} $\M$ is a functor from the category $(\mathcal{F}/\bfk)$ of finitely generated extensions of $\bfk$ to graded groups, such that $\M^{\bullet}(F)$ is a graded $\K^{\bullet}_{\M}(F)$-module, that is equipped with all of the maps and properties above and compatible with the respective maps and properties for Milnor's K-theory. The crucial examples to keep in mind are
\begin{itemize}
    \item $\K_{\ell}^{\bullet}(F) = \K_{\rm Mil}^{\bullet}(F)/\ell$ where $\ell$ is a positive integer coprime to ${\rm char}(\bfk)$.
    \item $\H_D^{\bullet}(F) = \oplus_i \H^i(F,D(i))$ where $D$ is a $\ell$-torsion Galois module over $\bfk$ and $D(i)=D\otimes \mu_\ell^{\otimes i}$.
\end{itemize}
The norm-residue isomorphism shows that $\K_\ell = \H_{\ZZ/\ell\ZZ}$, which also explains the $\K_{\rm Mil}$-module structure of $\H_{D}$ as coming from cup product.

\begin{rmk}\label{rmk:ellproduct}
Note that if $x \in \M^{\bullet}(F)$ is an $\ell$-torsion element and $\alpha \in \K^{\bullet}(F)$ then the product $\alpha \cdot x$ only depends on the class of $\alpha$ in $\K^{\bullet}_\ell(F)$. This shows that we have a well defined product
\[\K^{\bullet}_\ell (F) \times \M^{\bullet}(F)_\ell \to \M^{\bullet}(F)_{\ell}, \quad (\left[\alpha\right],x) \mapsto \alpha \cdot x\]
where $\alpha \in \K^{\bullet}(F)$ is any representative of $\left[ \alpha \right] \in \K^{\bullet}_{\ell}(F)$.
\end{rmk}

\subsection{Equivariant Chow groups with coefficients}\label{sec:equi chow}

Chow groups with coefficients were introduced by Rost \cite{Rost} and their equivariant counterpart was first used by Guillot \cite{Guil}, adapting ideas by Edidin-Graham and Totaro \cites{EG,Tot} to compute classical cohomological invariants. 

Let $X$ be a scheme of finite type over $\bfk$, and define 
\[C_i(X,\M)=\bigoplus_{V \subset X, {\rm dim}(V)=i} \M^{\bullet}(\bfk(V)) \]
where the sum is over all irreducible subvarieties of dimension $i$ of $X$. We have a complex
\[0 \leftarrow C_0(X,\M) \xleftarrow{\partial_1} C_1(X,\M) \xleftarrow{\partial_2} \ldots \xleftarrow{\partial_{{\rm dim}(X)}} C_{{\rm dim}(X)}(X,\M) \leftarrow 0 \]
where the differential is obtained from the ramification map $\partial_v$ we described earlier. We define the $i$-th dimensional Chow group with coefficients of $X$ as
\[{\rm A}_i(X, \M)= {\rm Ker}(\partial_i)/{\rm Im}(\partial_{i+1}). \]
If $X$ is equidimensional we can also order by codimension and define ${\rm A}^i(X,\M)={\rm A}_{{\rm dim}(X)-i}(X,\M)$.

The groups ${\rm A}_*(X,\M)$ and ${\rm A}^*(X,\M)$ are bigraded, with the two degrees coming from dimension (resp. codimension) and the grading of $\M$. We will always refer to the first grading as dimension or codimension and to the second grading as degree. When $\M=\K_{\rm Mil}$ the zero degree part of ${\rm A}_i(X,\K_{\rm Mil})$ retrieves the usual Chow group ${\rm CH}_i(X)$ and when $\M=\K_{\ell}$ the zero degree part of ${\rm A}_i(X,\K_\ell)$ is equal to ${\rm CH}_i(X)/\ell$.

Often we will need to keep track of a fixed degree component. To do that write
\[C_i(X,\M^d)=\bigoplus_{{\rm dim}(V)=i} \M^{d}(\bfk(V)),\quad \partial_i^d: C_i(X,\M^d) \to C_{i-1}(X,\M^{d-1})\]
so that we can define the quotient
\[{\rm A}_i(X, \M^d)= {\rm Ker}(\partial^d_i)/{\rm Im}(\partial_{i+1}^{d+1}) \]
as the component of dimension $i$ and degree $d$ of ${\rm A}_\bullet(X,\M)$. Correspondigly we define ${\rm A}^i(X,\M^d)={\rm A}_{{\rm dim}(X)-i}(X,\M^d)$. We will also often use the notation
\[ {\rm A}^i(X,\M^{\leq d}) = \bigoplus_{d' \leq d} {\rm A}^i(X,\M^{d'}).\]

Chow groups with coefficients enjoy all the properties of Chow groups and more. In the following list we'll often use codimensional notation because it's the notation that will most often appear in the rest of the paper.

\begin{enumerate}[label=\Alph*)]
    \item Given a morphism $f:Y \to X$ there's a functorial pullback map $f^i:{\rm A}^i(X,\M) \to {\rm A}^i(Y,\M)$ if $f$ is flat equidimensional or $X$ is smooth.
    \item Given a proper morphism $f:Y \to X$ there's a functorial pushforward map $f_*:{\rm A}_i(Y,\M) \to {\rm A}_i(X,\M)$.
    \item The pullback and pushforward map are compatible along cartesian diagrams.
    \item Given a closed immersion $i:V \to X$ of codimension $c$ with complement $U$ there's a ramification map $\partial_V:{\rm A}^*(U, \M) \to {\rm A}^*(V,\M)$ and a long exact sequence
\[
\ldots \to {\rm A}^{i}(X,\M) \to {\rm A}^{i}(U,\M) \xrightarrow{\partial_V} {\rm A}^{i-c+1}(V,\M) \xrightarrow{i_*} {\rm A}^{i+1}(X,\M) \to \ldots
\] 
    \item Given a vector bundle $E \to X$ there are Chern classes 
    \[c_i(E):{\rm A}^*(X, \M) \to {\rm A}^{*-i}(X,\M)\]
    satisfying the usal properties of Chern classes.
    \item Given a vector bundle $E \to X$ we have the projective bundle formula
    \[{\rm A}^{i}(\PP(E),\M)=\bigoplus_{r+s=i}c_1(\mathcal{O}_{\PP(E)}(-1))^r \cdot {\rm A}^{s}(X,\M).\]
    \item If $\M =\K_{\rm Mil}$ or $\M=\K_{\ell}$ and $X$ is smooth ${\rm A}^*(X,\M)$ forms a ring which is graded-commutative with respect to degree.
    \item If $X$ is smooth, then ${\rm A}^*(X, \M)$ is a ${\rm A}^*(X,\K_{\rm Mil})$-module. if moreover $\M$ is $\ell$-torsion then ${\rm A}^*(X, \M)$ is a ${\rm A}^*(X,\K_{\ell})$-module.
    \item The module structure above is compatible with pullbacks, pushfrowards and ramification maps. Moreover, there is a projection formula \[f_*(f^*(x)\cdot y)=x \cdot f_*(y).\]
\end{enumerate}

If our scheme is being acted upon by a smooth algebraic group $G/\bfk$ then we can form the equivariant Chow groups with coefficients by equivariant approximation, that is we take a representation $V$ of $G$ such that the action is free on an open subset $U$ whose complement has codimension greater than ${\rm dim}(X)-i$ and define
\[{\rm A}_{i}^G(X,\M)={\rm A}_{i+{\rm dim}(V)-{\rm dim}(G)}((X\times U)/G,\M).\]
Similarly if $X$ is equidimensional and the codimension of the complement of $U$ is greater than $i$ we define
\[{\rm A}_{G}^i(X,\M)={\rm A}^{i}((X\times U)/G,\M).\]
These groups do not depend on the choice of representation, and in fact do not even depend on the particular presentation we chose for $\Xcal=\left[X/G\right]$: if $\Xcal=\left[X/G\right]=\left[Y/G'\right]$ are two different presentations of the same stack as a quotient by a smooth algebraic group over $\bfk$ then we have 
\[
{\rm A}_i^{G}(X,\M)={\rm A}^{G'}_i(Y,\M)
\]
for all $i$. Thus we will sometimes write ${\rm A}_i(\Xcal,\M)$ or ${\rm A}^i(\Xcal,\M)$ when the chosen presentation is irrelevant.

Equivariant Chow groups with coefficients enjoy all the properties that ordinary Chow groups with coefficients do, as long as everything is $G$-equivariant.

\subsection{Cohomological invariants}\label{sec:coh}

In this subsection we will briefly recall the properties of cohomological invariants which will be needed throughout the paper. For an in-depth treatment of the subject the reader can refer to \cite{DilPirBr}*{Section 2}.

As mentioned in the introduction, given an algebraic stack $\Xcal/\bfk$ we consider \cite{DilPirBr}*{Definition 2.3} the functor 
\[{\rm Pt}_{\Xcal}: (\mathcal{F}/\bfk) \to ({\rm Set})\]
assigning to a finitely generated extension $F/\bfk$ the set $\Xcal(F)/\!\sim$ of $F$-valued points of $\Xcal$ modulo isomorphism. A cohomological invariant of $\Xcal$ will be a natural transformation
\[\alpha: {\rm Pt}_{\Xcal} \to \M\]
satisfying a \emph{continuity condition} which basically says that the natural transformation is well-behaved in regards to specializations. The cohomological invariants of $\Xcal$ with coefficients in $\M$ form a graded group $\Inv(\Xcal,\M)$. Cohomological invariants enjoy the following properties

\begin{enumerate}
    \item A morphism of algebraic stacks $f:\Ycal \to \Xcal$ induces a pullback map $f^*$ on cohomological invariants defined by $f^*(\alpha)(P) = \alpha(f(P))$.
    \item Cohomological invariants form a sheaf in the \emph{smooth-Nisnevich} topology, where coverings are smooth representable maps $\Xcal \to \Ycal$ such that any $F$-valued point of $\Ycal$ lifts to a $F$-valued point of $\Xcal$.
    \item In particular if $\M=\H_D$ and $\Xcal$ is smooth then cohomological invariants with coefficients in $\H_D$  are the sheafification of cohomology with coefficients in $D$ on the smooth-Nisnevich site of $\Xcal$.
    \item If $\Xcal$ is smooth over $\bfk$ an open immersion $\mathcal{U} \to \Xcal$ induces an injective map on cohomological invariants.
    \item If $\Xcal$ is smooth over $\bfk$ and $f:\Ycal \to \Xcal$ is either an affine bundle, the projectivization of a vector bundle or an open immersion whose complement has codimension $\geq 2$ then $f^*:\Inv(\Xcal, \M) \to \Inv(\Ycal, \M)$ is an isomorphism.
    \item If $\Xcal$ is smooth over $\bfk$ and $\Xcal_{\mathcal{D},\ell}$ is the order $\ell$ root stack of $\Xcal$ at an irreducible Cartier divisor $\mathcal{D}$ then the cohomological invariants of $\Xcal_{D,\ell}$ are the invariants of $\Xcal \smallsetminus \mathcal{D}$ whose ramification at $\mathcal{D}$ is of $\ell$-torsion.
    \item If $\Xcal=\left[ X/G\right]$ is the quotient of a smooth scheme by a smooth algebraic group over $\bfk$ then 
    \[
{\rm Inv}^1(\Xcal,\K_\ell)={\rm H}^1_{\textnormal{lis-\'et}}(\Xcal,\mu_\ell) \textnormal{ and } {\rm Inv}^2(\Xcal, {\rm H}_{\mu_\ell^{\vee}})={\rm Br}'(\Xcal)_{\ell}.
\]
    \item If $\Xcal=\left[ X/G\right]$ is the quotient of a smooth scheme by a smooth algebraic group over $\bfk$ then
    \[\Inv(\Xcal,\M)={\rm A}^0_{G}(X,\M).\]
\end{enumerate}
Therefore, computing cohomological invariants with values in a cycle module $\M$ of a smooth quotient stack $[X/G]$ with $G$ smooth over $\bfk$ is equivalent to computing the zero-codimension $G$-equivariant Chow group $A^0_G(X,\M)$. Cohomological invariants of a stack in turn give access to the Brauer group: this will be our main technical tool for the computation of the Brauer group of the moduli stacks of interest.

\section{Brauer groups of moduli of plane curves}
The main goal of this section is the computation of the Brauer group of $\Xcal_d$, the moduli stack of smooth plane curves of degree $d$, when the ground field is algebraically closed (\Cref{thm:brauer Xd}). This result is achieved via first computing the cohomological invariants of $\Xcal_d^\fr$, the moduli stack of \emph{framed} smooth plane curves (\Cref{prop: InvX4}). For this, we also need some preliminary results involving computations in the equivariant Chow groups, which are the focus of \Cref{sub:equi comp}.

Recall that, as mentioned in the notations, throughout the section we will assume that the characteristic of $\bfk$ is different from $2$ and does not divide $d$.
\subsection{Moduli stacks of smooth plane curves}\label{sub:plane curves}
The moduli stack $\Xcal_d$ of smooth plane curves of degree $d$ is the stack whose objects are given by the data 
\[(C\overset{\iota}{\hookrightarrow} P \to S)\]
where $P\to S$ is a Severi--Brauer variety whose fibers have dimension two (sometimes these varieties are called $\PP^2$-bundles), the scheme $C\to S$ is a smooth and proper morphism whose fibers have dimension one, and $\iota$ is a closed embedding whose restriction over any geometric point $s$ of $S$ realizes $C_s$ as a smooth plane curve of degree $d$ in $P_s\simeq \PP^2_s$. Recall from our notations that we will always assume that the characteristic of our base field $\bfk$ does not divide $d$.

Let $W_d$ be the vector space of trinary forms of degree $d$, so that $\PP(W_d)$ can be identified with the Hilbert scheme of plane curves of degree $d$ in $\PP^2$. Let $Z\subset\PP(W_d)$ be the divisor of singular forms. Then the moduli stack $\Xcal_d$ admits the following presentation as a quotient stack:
\[\Xcal_d \simeq [ \PP(W_d)\smallsetminus Z / \PGLt ],\]
where the action of $\PGLt$ on $\PP(W_d)$ is inherited from the natural action of $\GLt$ on $W_d$.

A related stack, useful for our computations, is the stack $\Xcal_d^\fr$ of \emph{framed} smooth plane curves of degree $d$. Its objects are pairs 
\[(E\to S,C\overset{\iota}{\hookrightarrow} \PP(E) \to S)\]
where $E\to S$ is a vector bundle of rank three, and $\iota$ is a closed embedding that over every geometric point $s$ of $S$ embeds $C_s$ in $\PP(E)_s\simeq\PP^2_s$ as a smooth plane curve of degree $d$. This stacks admits the following presentation:
\[ \Xcal_d^\fr := [(\PP(W_d)\smallsetminus Z)/\GLt], \]
where $\PP(W_d)$ inherits the $\GLt$-action from $W_d$.
\begin{rmk}
    The stack of framed smooth plane curves is of additional interest for us because, by \Cref{prop:presentation M3 minus H3}, we have that $\Mcal_3\smallsetminus\Hcal_3 \to \Xcal^{\fr}_4$ is a $\Gm$-torsor. This fact will allow us to compute the invariants of $\Mcal_3\smallsetminus\Hcal_3$ once those of $\Xcal^{\fr}_4$ are known.
\end{rmk}
The $\GLt$-invariant divisor $Z\subset \PP(W_d)$ can be further stratified. For our study, we will only deal with the following two $\GLt$-invariant, locally closed subschemes:
\begin{align*} 
Z_{\textnormal{irr}} &= \textnormal{subscheme of irreducible, singular curves of degree \textit{d}} \\
Z_{\binod} &= \textnormal{subscheme of irreducible curves of degree \textit{d} with exactly two nodes.}
\end{align*}
The subscheme $Z_{\textnormal{irr}}$ is open in $Z$, and $\overline{Z_{\binod}}$ has codimension one in $Z$. The divisor $Z\subset \PP(W_d)$ admits the following desingularization: let $\widetilde{Z}\subset \PP(W_d)\times \PP^2$ be the singular locus of the universal curve over $\PP(W_d)$. More explicitly, we have
\[ \widetilde{Z} = \left\{ ([f],p)\textnormal{ such that }\partial_{x_i} f (p) = 0,\textnormal{ for }i=0,1,2. \right\}\]
where $[f]$ is a point of $\PP(W_d)$ and $p$ is a point of $\PP^2$. 
\begin{prop}
The two projections
\[ \widetilde{Z} \longrightarrow \PP^2, \quad \widetilde{Z} \longrightarrow \PP(W_d) \]
realizes $\widetilde{Z}$ respectively as a projective bundle over $\PP^2$, and as a birational model of $Z$. In particular, the induced morphism $\widetilde{Z} \to Z$ is an isomorphism over $Z_{\textnormal{irr}}\smallsetminus Z_{\binod}$ and is \'{e}tale of degree two over $Z_{\binod}$.
\end{prop} 
We will denote $\widetilde{Z}_\irr$ (respectively $\widetilde{Z}_\binod$) the preimage of $Z_\irr$ (respectively $Z_\binod$) along $\widetilde{Z} \to Z$. The closure of $\widetilde{Z}_\binod$ in $\widetilde{Z}$ will be denoted $\widetilde{Z}_\binod^*$.
\subsection{Computation of equivariant fundamental classes}\label{sub:equi comp}
To compute the (low degree) cohomological invariants of $\Xcal_d^\fr$, we will need explicit expressions of $[Z]_{\GLt}$ in ${\rm CH}^*_{\GLt}(\PP(W_d)\times\PP^2)$ and of $[\widetilde{Z}_\binod^*]_{\GLt}$ in ${\rm CH}^*_{\GLt}(\widetilde{Z})$. For a brief review of the basic tools of equivariant intersection theory that are used below, the reader can consult \cite{DilFulVis}*{Section 2}.
Recall that the $\GLt$-equivariant Chow ring of a point is 
\[{\rm CH}_{\GLt}(\Spec(\bfk)) \simeq \ZZ[c_1,c_2,c_3] \]
where $c_i$ is the $i^{\rm th}$ equivariant Chern class of the standard $\GLt$-representation.

Let $T\subset\GLt$ the torus of diagonal matrices. Then the $T$-equivariant Chow ring of a point is 
\[ {\rm CH}_T(\Spec(\bfk)) \simeq \ZZ[l_1,l_2,l_3]\]
where $l_i$ is the $T$-equivariant first Chern class of the rank one representation determined by the formula
$(u_1,u_2,u_3)\cdot x := u_i^{-1}x$.
With this choice of the generators, the embedding of equivariant Chow rings $${\rm CH}_{\GLt}(\Spec(\bfk)) \simeq \ZZ[c_1,c_2,c_3]\longrightarrow {\rm CH}_T(\Spec(\bfk))$$ is given by  $$c_1\mapsto -(l_1+l_2+l_3), \quad c_2\mapsto l_1l_2+l_1l_3+l_2l_3, \quad c_3\mapsto-l_1l_2l_3.$$
Therefore, when working with $T$-equivariant Chow rings, the classes $c_i$ denote the cycles above. 
%Beware that our notation here differs slightly from the one used in \cite{DilFulVis}*{Section 2}, where instead of $\ell_i$ those classes are called $l_i$.
\subsubsection{The fundamental class of $Z$}\label{sec:class Z}
The locus $\widetilde{Z}\subset \PP(W_d)\times \PP^2$ is a complete intersection of the three $T$-invariant hypersurfaces
\[ H_i := \{([f],p) \text{ such that }\partial_{x_i}f(p)=0\}, \quad i=1,2,3. \]
The $T$-equivariant fundamental class of $H_i$ is then \cite{EF}*{Lemma 2.4}
\[ [H_i]_T = (h+(d-1)t+l_i) \]
where $h$ (respectively $t$) is the hyperplane class of $\PP(W_d)$ (respectively $\PP^2$). It follows then that
\begin{align*}
    [\widetilde{Z}]_T&= \prod_{i=1}^{3}(h+(d-1)t+l_i) \\
    &= (d-1)^3t^2 + (d-1)^2t^2 (3h + l_1 + l_2 + l_3) + t\cdot p + q,
\end{align*}
where $p$ and $q$ are polynomials in $h$ and $l_i$. Using the fact that $\pr_1|_{\widetilde{Z}}\colon\widetilde{Z}\to Z$ is generically an isomorphism, together with the formulas 
$$\pr_{1*}(1)=\pr_{1*}(t)=0, \quad \pr_{1*}(t^2)=1,\quad \pr_{1*}(t^3)=-c_1 $$
we obtain
\begin{align*}
    [Z]_T &= \pr_{1*}[\widetilde{Z}]_T \\
    &= -(d-1)^3c_1 + (d-1)^2(3h-c_1) \\
    &= (d-1)^2(3h-dc_1).
\end{align*}
We have proven the following.
\begin{lm}\label{lm:class Z}
    The $\GLt$-equivariant fundamental class of $Z$ in the equivariant Chow ring of $\PP(W_4)$ is equal to
    \[ [Z]_{\GLt} = (d-1)^2(3h-dc_1).  \]
\end{lm}
\subsubsection{The fundamental class of $\widetilde{Z}_\binod^*$}
This section is devoted to the computation of the $\GLt$-equivariant fundamental class of $\widetilde{Z}_\binod^*$ in ${\rm CH}_{\GLt}(\widetilde{Z})$. 

The $\GLt$-action on $\PP(W_d)\times\PP^2\times\PP^2$ can be used to define a $T$-action on the same scheme. We have
\[ {\rm CH}_{T}(\PP(W_4)\times\PP^2\times\PP^2) \simeq \ZZ[l_1,l_2,l_3,h,s,t]/I, \]
where $h$ (respectively $s$, $t$) is the hyperplane section of $\PP(W_d)$ (respectively of the first $\PP^2$, of the second $\PP^2)$.

For $i=1,2,3$, let $H_i\subset \PP(W_d)\times\PP^2\times\PP^2$ be the hypersurface of equation
\[ H_i:=\{([f],p,q) \text{ such that }\partial_{x_i}f(p)=0\},\]
where $x_1,x_2$ and $x_3$ are the variables of the trinary form $f$.
Similarly, we define the hypersurface $F_i \subset \PP(W_4) \times \PP^2 \times \PP^2$ as
\[ F_i:=\{([f],p,q) \text{ such that }\partial_{x_i}f(q)=0\}.\]
We have \cite{EF}*{Lemma 2.4}
\[ [H_i]_T = (h+(d-1)s+l_i),\quad [F_i]_T=(h+(d-1)t+l_i).\]
Set $Y=\left(\cap_{i=1}^{3} H_i\right)\cap\left(\cap_{i=1}^{3} F_i\right)$ to be the set-theoretic intersection, and let
\[ \xi := \prod_{i=1}^{3} (h+(d-1)s+l_i) \cdot \prod_{i=1}^{3}(h+(d-1)t+l_i) \]
be the intersection-theoretic product. Observe that $Y$ is supported on the union $B\cup R$, where
\begin{itemize}
    \item the scheme $B$ is the locus of triples $([f],p,q)$ such that $[f]$ belongs to $\overline{Z}_{\binod}$ and $p$, $q$ are singular points;
    \item the scheme $R$ is the locus of triples $([f],p,q)$ where $p=q$, or equivalently is the image of $\widetilde{Z}$ through the morphism ${\rm id}\times\delta:\PP(W_4)\times\PP^2\to \PP(W_4)\times\PP^2\times\PP^2$.
\end{itemize}
Therefore, as the intersection is transversal on $B$, we have $\xi=[B]_T+\rho$ where $\rho$ is a cycle supported on $R$. The cycle $\rho$ can be computed via a simple version of the residual intersection formula \cite{Ful}*{Proposition 9.11}. Consider the closed embedding 
\[ \iota: \widetilde{Z} \overset{j}{\longrightarrow} \PP(W_d) \times \PP^2 \overset{{\rm id}\times \delta}{\longrightarrow} \PP(W_d) \times \PP^2 \times \PP^2. \]
For what follows, we set the following notation:
\begin{itemize}
    \item the hyperplane class of $\widetilde{Z}$, regarded as a projective bundle over $\PP^2$, is denoted $h_{\widetilde{Z}}$;
    \item the hyperplane class of $\PP^2$, pulled back along $\widetilde{Z}\to\PP^2$, is denoted $u$.
\end{itemize}
With this notation, we see that $\iota^*h=h_{\widetilde{Z}}$ and $\iota^*s=\iota^*t=u$. Then the residual intersection formula tells us that
\begin{align*}
    \rho &= \iota_*\{ \prod_{i=1}^{3} c^T(N_{H_i}|_{\widetilde{Z}})\cdot \prod_{i=1}^{3} c^T(N_{F_i}|_{\widetilde{Z}}) \cdot c^T(N_{\iota})^{-1}\}_1 \\
    &= \iota_*(\sum_{i=1}^{3}( c^T_1(N_{H_i}|_{\widetilde{Z}})+c^T_1(N_{F_i}|_{\widetilde{Z}}))-c^T_1(N_{\iota})) \\
    &= \iota_*(2(\sum_{i=1}^{3} h_{\widetilde{Z}} + (d-1)u + l_i)-c_1^T(N_{\iota}))
\end{align*}
where $N_{\iota}$ is the normal bundle relative to the closed embedding $\iota:\widetilde{Z} \hookrightarrow \PP(W_4)\times\PP^2\times\PP^2$. Using the exact sequence of normal bundles
\[ 0 \longrightarrow N_{j} \longrightarrow N_{\iota} \longrightarrow j^*N_{{\rm id}\times\delta} \longrightarrow 0 \]
we obtain
\begin{align*}
     c_1^T(N_{\iota}) &= c_1^T(N_{j}) + j^*(c_1^T(N_{{\rm id}\times\delta}))\\
     &=(\sum_{i=1}^{3} h_{\widetilde{Z}} + (d-1)u + l_i) + 3u + l_1+l_2+l_3.
\end{align*}
In the last equality above we used the fact that for any smooth variety $X$, the diagonal embedding $\delta\colon X \hookrightarrow X\times X$ is a regular embedding and we have $N_{\delta}\simeq \Omega_X^\vee$, and $c_1^G(\Omega_X^\vee)= -c_1^G(\det(\Omega_X))$. This, applied to $X=\PP^2$, gives us $c_1^T(N_\delta)=-c_1^T(\omega_{\PP^2})=3u+l_1+l_2+l_3$.

From this we deduce
\begin{align*} 
\rho = \iota_*(3h_{\widetilde{Z}} + 3(d-2)u)  = \iota_*(\iota^*(3h + 3(d-2)s)) = (3h+3(d-2)s)\cdot [R]_T.
\end{align*}
This implies that 
\begin{align*} 
[B]_T &= \xi - \rho =  \prod_{i=1}^{3} (h+(d-1)s+l_i) \cdot \prod_{i=1}^{3}(h+(d-1)t+l_i) - (3h+3(d-2)s)\cdot [R]_T \\
&= \pr_{12}^*[\widetilde{Z}]_T\cdot ( \prod_{i=1}^{3}(h+(d-1)t+l_i)) - (3h+3(d-2)s)\cdot [R]_T,
\end{align*}
where $\pr_{12}:\PP(W_d)\times\PP^2\times\PP^2\to \PP(W_d)\times\PP^2$ is the projection on the first two factors. Observe that $\pr_{12*}([B]_T)=[\widetilde{Z}_\binod^* \subset \PP(W_d)\times \PP^2]_T$, i.e. the equivariant fundamental class of $\widetilde{Z}_\binod^*$ in the Chow ring of $\PP(W_d)\times\PP^2$. Using the previous computation of $[B]_T$, and applying the projection formula, we get
\begin{align*}
    \pr_{12*}[B]_T &= [\widetilde{Z}]_T\cdot \pr_{12*}(\prod_{i=1}^{3}(h+(d-1)t+l_i)) - (3h+3(d-2)s)\pr_{12*}(\iota_*[\widetilde{Z}]_T)\\
    &= [\widetilde{Z}]_T\cdot (\pr_{12*}(\prod_{i=1}^{3}(h+(d-1)t+l_i)) - (3h+3(d-2)s)).
\end{align*}
Observe now that we have $\pr_{12*}(1)=\pr_{12*}(t)=0$, $\pr_{12*}(t^2)=1$. Moreover, using the relation $t^3+c_1t^2+c_2t+c_3=0$, we deduce that $\pr_{12*}(t^3)=-c_1$. Therefore we get
\begin{align*}
     \pr_{12*}(\prod_{i=1}^{3}(h+(d-1)t&+l_i)) - (3h+3(d-2)s)= \\ 
     &=3((d-1)^2-1)h - ((d-1)^3+(d-1)^2)c_1 - 3(d-2)s.
\end{align*}
Let $[\widetilde{Z}_{\binod}^*]_T$ be the fundamental class of $\widetilde{Z}_\binod^*$ in the Chow ring of $\widetilde{Z}$, so that we have $[\widetilde{Z}_\binod^*]_T=p(h_{\widetilde{Z}},u,c_1)$. As the pullback homomorphism 
\[j^*: {\rm CH}_T^*(\PP(W_d)\times\PP^2) \longrightarrow {\rm CH}_T^*(\widetilde{Z}) \]
is surjective, it follows that
\[[\widetilde{Z}_\binod^* \subset \PP(W_d)\times\PP^2 ]_T = j_*[\widetilde{Z}_\binod^*]_T = p(h,s,c_1)\cdot [\widetilde{Z}]_T. \]
It is easy to check that there is no non-zero polynomial $q(h,s,c_1)$ of degree one such that $q\cdot [\widetilde{Z}]_T =0$: indeed, we have $[\widetilde{Z}]_T=h^3+f(h,s,c_1,c_2,c_3)$ (see \Cref{sec:class Z}), where $f$ has degree two with respect to the variable $h$. The subgroup of relations in ${\rm CH}_T^4(\PP(W_d)\times\PP^2)$ is generated by the cycles
\[ h(s^3+c_1s^2+c_2s+c_3),\quad s(s^3+c_1s^2+c_2s+c_3), \quad  l_i(s^3+c_1s^2+c_2s+c_3) \text{ for }i=1,2,3. \]
In particular, there is no relation in degree four containing a monomial divisible by $h^3$, from which follows that $q\cdot [\widetilde{Z}]_T=0$ if and only if $q=0$.

As we have proved that $$p(h,s,c_1)\cdot [\widetilde{Z}]_T = (3((d-1)^2-1)h - ((d-1)^3+(d-1)^2)c_1 - 3(d-2)s) \cdot [\widetilde{Z}]_T$$ where $p(h_{\widetilde{Z}},u,c_1)=[\widetilde{Z}_\binod^*]_T$, we finally deduce the following.
\begin{lm}\label{lm:class D}
    The $\GLt$-equivariant fundamental class of $\widetilde{Z}_\binod^*$ in the equivariant Chow ring of $\widetilde{Z}$ is equal to
    \[ [\widetilde{Z}_{\binod}^*]_{\GLt} = (3d(d-2)h_{\widetilde{Z}} - d(d-1)^2c_1 - 3(d-2)u). \]
\end{lm}

\subsection{Cohomological invariants of $\Xcal_d^\fr$ in low degree}

%In this subsection we will go as far with computing the cohomological invariants of $\Mcal_3$ as it is possible using equivariant Chow rings with coefficients. It will turn out to be enough to compute them over an algebraically closed field; the explicit construction in the next section will conclude our computation for an abitrary base field of characteristic different from $2$. 

%We are going to proceed from bottom to top, starting from the highest codimensional part of the stratification of $\Xcal_{4}$, working our way up to (almost) understanding the invariants of $\Xcal_4$ and then turning our attention to the $\Gm$-torsor $\Mcal_3 \smallsetminus \Hcal_3 \to \Xcal_4$. First we need a lemma on ${\rm GL}_3$-equivariant Chow groups.
In this section, we study the low degree cohomological invariants of $\Xcal_d^{\fr}$ with coefficients in a cycle module $\M$, leveraging their interpretation as the codimension zero $\GLt$-equivariant Chow group of $\PP(W_d)\smallsetminus Z$ with coefficients in $\M$. We start with a simple lemma.
\begin{lm}[\cite{DilPirBr}*{Proposition 4.2}]
Let $\mathcal{E}=\left[ \bA^3/{\rm GL}_3\right] \to {\rm BGL}_3$ be the standard representation of ${\rm GL}_3$ and let $c_i = c_i(\mathcal{E})$ for $i = 1,2,3$. Then 
\[{\rm A}^{\bullet}_{{\rm GL}_3}(\Spec(\bfk),\M)=\ZZ\left[c_1, c_2, c_3 \right] \otimes \M^{\bullet}(\bfk).\]
\end{lm}

In particular, if $X$ is a tower of ${\rm GL}_3$-equivariant affine bundles and ${\rm GL}_3$-equivariant projective bundles, we can apply the lemma and the projective bundle formula to completely reduce its equivariant Chow groups with coefficients to the ordinary (equivariant) Chow groups and the chosen cycle module.

\begin{lm}\label{lm:tensor}
Let $X$ be a ${\rm GL}_3$-equivariant scheme that is obtained from $\Spec(\bfk)$ as a sequence of $\GLt$-equivariant affine and projective bundles. Then
\[{\rm A}^{\bullet}_{\GLt}(X)={\rm CH}^{\bullet}_{\GLt}(X) \otimes \M^{\bullet}(\bfk).\]
\end{lm}
\begin{proof}
The statement follows easily by starting from ${\rm A}^{\bullet}_{{\rm GL}_3}(\Spec(\bfk),\M)$ and applying at each step either the projective bundle formula or homotopy invariance.  
\end{proof}

This applies to $\widetilde{Z}$, which by \cite{FulVis}*{Proposition 4.2} is an equivariant projective bundle over $\PP^2$. Using this we can deal with the divisor $Z$ of singular plane curves.

\begin{rmk}\label{rmk:splitting} 
Often the Chow groups with coefficients we work with will exhibit surprisingly good behavior, working similarly to free modules and allowing us to easily split exact sequences. The following two cases will occur repeatedly throughout the paper:
\begin{enumerate}
    \item Consider a smooth algebraic group $G$ acting on a scheme $X$, and let $U \subseteq X$ be its smooth locus. If $U$ has a $\bfk$-rational point, then the composition
    \[{\rm A}^0_G(X,\M) \to {\rm A}^0(U,\M) \to A^0(\Spec(\bfk),\M)=\M^{\bullet}(\bfk)\]
    is a retraction of the pullback $\M^{\bullet}(\bfk)=A^0(\Spec(\bfk),\M) \to A^0_G(X,\M)$ so that $\M^{\bullet}(\bfk)$ is a direct summand in $A^0_G(X,\M)$.

    Note that if the field $\bfk$ is finite (which is often a problem when looking for points on an open subset) we can replace it with $\bfk(t)$ as the map
    \[x \mapsto \partial_{t = 0}(\lbrace t \rbrace \cdot x)\]
    is a retraction from $\M^{\bullet}(\bfk(t))$ to $\M^{\bullet}(\bfk)$.
    \item Consider a $G$-equivariant closed immersion $D \to X$ of codimension one. Assume that there exist positive integers $\ell_i \mid \ell$ and elements $\alpha_i\in {\rm A}^0_G(D,\K_{\ell_i})$ such that for all $\ell$-torsion modules we have:
    \begin{enumerate}
        \item The map $\M^{\bullet}_{\ell_i}(\bfk) \to {\rm A}^{\bullet}(D, \M)$ given by $x \mapsto \alpha_{i}\cdot x$ is injective for all $i$.
        \item An equality
        \[{\rm Im}(\partial_D: {\rm A}^{0}_G(X\smallsetminus D,\M)\to {\rm A}^{0}_G(D,\M)) = \oplus_i \alpha_i\cdot\M^{\bullet}(\bfk)_{\ell_i}. \]
    \end{enumerate}
        Then picking $\M=\K_{\ell_i}$ we can find elements $\beta_i \in {\rm A}^0_G(X\smallsetminus D,\K_{\ell_i})$ such that $\partial_D(\beta_i)=\alpha_i$ and the map
    \[\oplus_i \alpha_i\M^{\bullet}(\bfk)_{\ell_i} \to A^{0}_G(X \smallsetminus D, \M), \quad \alpha_i \cdot x \mapsto \beta_i \cdot x  \]
    is a splitting of $\partial_D$ as by the compatiblity of product and boundary map \cite{Rost}*{R3f, p.330} we have $\partial_D(\beta_i \cdot x)=\alpha_i \cdot x$.
\end{enumerate}
     
\end{rmk}

\begin{prop}
We have 
\[  {\rm A}^0_{{\rm GL}_3}(Z, \M^{\leq 1}) = {\rm A}^0_{{\rm GL}_3}(Z_{\textnormal{irr}}, \M^{\leq 1}) =\begin{cases} \M^{\leq 1}(\bfk) \oplus \M^{0}(\bfk)_2\left[1\right] \,& \textnormal{if } d \textnormal{ is even,}\\
\M^{\leq 1}(\bfk) \,& \textnormal{if }  d \textnormal{ is odd.}
\end{cases}\]
where $\M^{0}(\bfk)_2\left[1\right]= \alpha\cdot\M^{0}(\bfk)_2$ for an element $\alpha\in {\rm A}^0_{{\rm GL}_3}(Z_{\textnormal{irr}},\K_2)$ of degree one. 
\end{prop}
\begin{proof}
The first equality follows from the fact that the complement of $Z_\irr$ in $Z$ has codimension $\geq 2$.

Recall that as $\widetilde{Z}$ is obtained as a sequence of projective bundles over $\Spec(\bfk)$ we have
\[{\rm A}^{\bullet}_{\GLt}(\widetilde{Z},\M)={\rm CH}_{\GLt}^{\bullet}(\widetilde{Z})\otimes \M(\bfk)\]
so in particular ${\rm A}_{{\rm GL}_3}^0(\widetilde{Z}, \M)=\M^{\bullet}(\bfk)$. As the complement of $\widetilde{Z}_\irr$ has codimension $2$ in $\widetilde{Z}$ we must have that ${\rm A}_{{\rm GL}_3}^0(\widetilde{Z}_\irr, \M)$ is trivial as well.

The projection $\widetilde{Z}_\irr \to Z_\irr$ induces the following commutative diagram of equivariant Chow groups with coefficients
\[
\xymatrixcolsep{1pc}\xymatrix{
 {\rm A}_{{\rm GL}_3}^0(\widetilde{Z}_\irr, \M) \ar[d] \ar[r] & {\rm A}_{{\rm GL}_3}^0(\widetilde{Z}_\irr \smallsetminus \widetilde{Z}_{\binod}^*, \M) \ar@{=}[d] \ar[r]^{\partial} & {\rm A}_{{\rm GL}_3}^0(\widetilde{Z}_{\binod}^*, \M)  \ar[d]^{2:1} \ar[r]^{i_*} & {\rm A}_{{\rm GL}_3}^1(\widetilde{Z}_\irr, \M)\\
 {\rm A}_{{\rm GL}_3}^0(Z_\irr, \M) \ar[r] & {\rm A}_{{\rm GL}_3}^0(Z_\irr \smallsetminus \overline{Z}_\binod, \M)  \ar[r]^{\partial} & {\rm A}_{{\rm GL}_3}^0(\overline{Z}_\binod, \M)  
 }
 \] 
We are only interested in the cohomological invariants of $Z_\irr$ of cohomological degree $\leq 1$. 

From the diagram above, and the fact that $\widetilde{Z}_\irr \smallsetminus \widetilde{Z}_{\binod}^*$ has a point, we have that ${\rm A}_{{\rm GL}_3}^0(\widetilde{Z}_\irr \smallsetminus \widetilde{Z}_{\binod}^*, \M^{\leq 1})$ splits as the direct sum of ${\rm A}_{{\rm GL}_3}^0(\widetilde{Z}_\irr, \M^{\leq 1})$ and
\[\ker(i_*\colon{\rm A}_{{\rm GL}_3}^0(\widetilde{Z}_{\binod}^*, \M^0) \to {\rm A}_{{\rm GL}_3}^1(\widetilde{Z}_\irr, \M^0) ).\]
Observe that, as $\widetilde{Z}_{\binod}^*$ is irreducible, we have ${\rm A}_{{\rm GL}_3}^0(\widetilde{Z}_{\binod}^*, \M^0) \simeq \M^0(\bfk)$.

The pushforward $i_*$ sends any $x \in \M^{\bullet}(\bfk)$ to $[\widetilde{Z}_{\binod}^*]_{\GLt}\cdot x \in {\rm A}^1_{{\rm GL}_3}(\widetilde{Z}_\irr, \M)$. The complement of $\widetilde{Z}_{\rm irr}$ in $\widetilde{Z}$ has codimension $2$, thus the pullback 
\[ A^{1}_{\GLt}(\widetilde{Z},\M) \to A^{1}_{\GLt}(\widetilde{Z}_{\rm irr},\M)
\]
is injective. As $[\widetilde{Z}_{\binod}^*]_{\GLt}\cdot x$ comes from the left hand side, we can check whether it is zero there. Then the element $[\widetilde{Z}_{\binod}^*]_{\GLt}\cdot x$ is zero if and only if $[\widetilde{Z}_{\binod}^*]_{\GLt}\otimes x$ is zero in ${\rm CH}^1_{\GLt}(\widetilde{Z}_\irr)\otimes\M^{\bullet}(\bfk)$. As the group on the left of the tensor product is finitely generated and free, all that matters is how divisible $[\widetilde{Z}_{\binod}^*]_{\GLt}$ is (because the subgroup generated by a primitive element is a direct summand in ${\rm CH}^1_{\GLt}(\widetilde{Z}_\irr)$). By \Cref{lm:class D} we conclude that
\begin{align*}
    {\rm A}_{{\rm GL}_3}^0(\widetilde{Z}_\irr \smallsetminus \widetilde{Z}_{\binod}^*, \M^{\leq 1})=\M^{\leq 1}(\bfk) \oplus \M^{0}(\bfk)_{r}\left[1\right].
\end{align*}
Here $r={\rm gcd}(d(d-1)^2,3(d-2))$, and the $\M^0(\bfk)_r\left[1\right]$ component is equal to $\beta \cdot \M^{0}(\bfk)_r$, where $\beta \in {\rm A}_{{\rm GL}_3}^0(\widetilde{Z}_\irr \smallsetminus \widetilde{Z}_{\binod}^*, \K_r)$ is an inverse image of $1 \in {\rm A}_{{\rm GL}_3}^0(\widetilde{Z}_{\binod}^*, \K_r)$. We have a direct sum as explained in \ref{rmk:splitting}, case (2).

Now we need to understand 
\[{\rm A}_{{\rm GL}_3}^0(Z_{\textnormal{irr}}, \M^{\leq 1})={\rm Ker}({\rm A}_{{\rm GL}_3}^0(Z_{\textnormal{irr}} \smallsetminus \overline{Z}_{\binod}, \M^{\leq 1})  \longrightarrow {\rm A}_{{\rm GL}_3}^0(\overline{Z}_{\binod}, \M^{\leq 1})).\]
We know that ${\rm A}_{{\rm GL}_3}^0(Z_{\textnormal{irr}} \smallsetminus \overline{Z}_{\binod}, \M)={\rm A}_{{\rm GL}_3}^0(\widetilde{Z} \smallsetminus \widetilde{Z}_{\binod}^*, \M)$ and that the kernel of 
\[{\rm A}_{{\rm GL}_3}^0(\widetilde{Z}_\irr \smallsetminus \widetilde{Z}_{\binod}^*, \M) \longrightarrow {\rm A}_{{\rm GL}_3}^0(\widetilde{Z}_{\binod}^*, \M)\]
is exactly $\M^{\bullet}(\bfk)$, so by the compatibility of pushforward and ramification map we conclude that
\[{\rm A}_{{\rm GL}_3}^0(Z_{\textnormal{irr}}, \M^{\leq 1}) = \M^{\leq 1}(\bfk) \oplus {\rm Ker}(\M^{0}(\bfk)_r \longrightarrow {\rm A}_{{\rm GL}_3}^0(\overline{Z}_{\binod}, \M^{0}))\left[1\right].\]
The map $\M^{0}(\bfk)_r \to {\rm A}_{{\rm GL}_3}^0(\overline{Z}_{\binod}, \M^{0})$ is just the restriction to $\M^{0}(\bfk)_r$ of multiplication by $2$ on $\M^{\bullet}(\bfk)$, so 
\[{\rm Ker}(\M^{0}(\bfk)_r \longrightarrow {\rm A}_{{\rm GL}_3}^0(\overline{Z}_{\binod}, \M^{0}))= 
\begin{cases}
    \M^{0}(\bfk)_2 & \textnormal{if } d \textnormal{ is even,}\\
    0 & \textnormal{if } d \textnormal{ is odd}\\ 
\end{cases} \]
where in the first case the kernel is generated by $\alpha = (r/2)\beta$.
\end{proof}

Now we are ready to compute the low degree cohomological invariants of $\Xcal^{\fr}_d$.

\begin{prop}\label{prop: InvX4}
Set $i_d=1$ when $3|d$, and zero otherwise. Then for $d$ even we have \[{\rm A}^0_{{\rm GL}_3}(\PP(W_d)\smallsetminus Z,\M^{\leq 2})=\M^{\leq 2}(\bfk) \oplus \beta_1 \cdot \M^{\leq 1}(\bfk)_{3^{i_d}(d-1)^2} \oplus N\left[2\right]\]
where $N \subseteq \M^{0}(\bfk)_2$ is the kernel of 
\[i_*:\alpha\cdot \M^{0}(\bfk)_2 \subset {\rm A}^0_{{\rm GL}_3}(Z,\M^2) \longrightarrow {\rm A}^1_{{\rm GL}_3}(\PP(W_d),\M^1).\]
For $d$ odd, we have
\[{\rm A}^0_{{\rm GL}_3}(\PP(W_d)\smallsetminus Z,\M^{\leq 2})=\M^{\leq 2}(\bfk) \oplus \beta_1 \cdot \M^{\leq 1}(\bfk)_{3^{i_d}(d-1)^2}. \]
\end{prop}
\begin{proof}
The claim is an immediate consequence of the exact sequence
\[
0 \to {\rm A}_{{\rm GL}_3}^0(\PP(W_d), \M) \to {\rm A}_{{\rm GL}_3}^0(\PP(W_d)\smallsetminus Z, \M) \to {\rm A}_{{\rm GL}_3}^0(Z, \M) \to {\rm A}_{{\rm GL}_3}^1(\PP(W_d), \M)
\]
and the fact that by \Cref{lm:class Z} the class $\left[ Z \right]_{\GLt}$ is divisible by $3^{i_d}(d-1)^2$ in ${\rm A}_{{\rm GL}_3}^1(\PP(W_d), \M)\simeq {\rm CH}^1_{\GLt}(\PP(W_d))\otimes \M^{\bullet}(\bfk)$. The splitting of the exact sequence is as in Remark \ref{rmk:splitting}, case (1). We know there is always a point in $\PP(W_d)\smallsetminus Z$ thanks to \cite{PoonHyp}.
\end{proof}
Recall that for a smooth quotient stack $[X/G]$, we have an identification
\[\Inv([X/G], \M) \simeq {\rm A}^0_G(X,\M). \]
Moreover, note that if $\bfk$ is algebraically closed and $\M$ is equal to $\K_\ell$ or $\H_D$ for some $\ell$-torsion Galois module there are no elements of positive degree in $\M^{\bullet}(\bfk)$ and consequently by the projective bundle formula in ${\rm A}^1_{{\rm GL}_3}(\PP(W_4),\M)$, so we can conclude that $N=\M^{0}(\bfk)_2$. Then using \Cref{prop: InvX4} we can easily deduce the following:
\begin{thm}\label{thm:inv Xdfr}
    Assume that the ground field $\bfk$ is an algebraically closed field whose characteristic is not $2$ and does not divide $d$, and that $\M$ is equal to $\K_\ell$ or $\H_D$ for some $\ell$-torsion Galois module.
    Set $i_d=1$ if $3|d$, and zero otherwise. Then we have
    \[{\rm Inv}^{\leq 2}(\Xcal_d^\fr,\M) \simeq
    \begin{cases}
        \M^{\leq 2}(\bfk) \oplus \M^{\leq 1}(\bfk)_{3^{i_d}(d-1)^2}[1] \oplus \M^0(\bfk)_2[2] & \textnormal{if } d \textnormal{ is even,}\\\\
        \M^{\leq 2}(\bfk) \oplus \M^{\leq 1}(\bfk)_{3^{i_d}(d-1)^2}[1] & \textnormal{if } d \textnormal{ is odd.}\\
    \end{cases} \]
    Moreover, for $d$ odd the presentation above holds for any field without the assumption of $\bfk$ being algebraically closed and for any cycle module $\M$. 
\end{thm}
\subsection{Brauer group of moduli of smooth plane curves}\label{sec:brauer Xd}
As an immediate application of \Cref{thm:inv Xdfr} we obtain the following result on the Brauer group of $\Xcal_d^\fr$.
\begin{cor}\label{cor:brauer Xdfr}
Let $\bfk$ be an algebraically closed field whose characteristic is not $2$ and does not divide $d$. Then
\begin{enumerate}
    \item if ${\rm char}(\bfk)=0$, we have ${\rm Br}(\Xcal_d^\fr) \simeq \ZZ/\langle 2,d\rangle$.
    \item if ${\rm char}(\bfk)=p>0$, we have ${\rm Br}(\Xcal_d^\fr) \simeq \ZZ/\langle 2, d\rangle\oplus B_{d,p}$, where $B_{d,p}$ is of $p$-primary torsion.
\end{enumerate}
\end{cor}

\begin{rmk}
If $d$ is odd we can describe the Brauer group of $\Xcal_d^{\fr}$ without assuming that $\bfk$ is algebraically closed:
\begin{enumerate}
    \item if ${\rm char}(\bfk)=0$, we have ${\rm Br}(\Xcal_d^\fr) \simeq {\rm Br}(\bfk) \oplus {\rm H}^1(\bfk,\ZZ/3^{i_d}(d-1)^2\ZZ)$.
    \item if ${\rm char}(\bfk)=p>0$, we have ${\rm Br}(\Xcal_d^\fr)_\ell \simeq {\rm Br}(\bfk)_\ell \oplus {\rm H}^1(\bfk,\ZZ/3^{i_d}(d-1)^2\ZZ)_\ell$.
\end{enumerate}
the ${\rm H}^1(\bfk,\ZZ/3^{i_d}(d-1)^2\ZZ)$ component comes from the cup products of $$\beta_1 \in {\rm Inv}^1(\Xcal_d^\fr,\K_{3^{i_d}(d-1)^2})={\rm H}^1(\Xcal_d^\fr,\mu_{3^{i_d}(d-1)^2})$$ and elements of ${\rm H}^1(\bfk,\ZZ/{3^{i_d}(d-1)^2}\ZZ)$. The splitting is because $\Xcal_d^{\fr}$ always has a $\bfk$-rational point.
\end{rmk}

Recall from \Cref{sub:plane curves} that $\Xcal_d^{\fr}$ is a $\Gm$-gerbe over the moduli stack of smooth plane curves $\Xcal_d$. We want to leverage this result to get to the Brauer group of $\Xcal_d$. First, let us recall the following.
\begin{thm}[\cite{ShiBr}*{Theorem 1.2}]
    Let $S$ be a scheme, and let $\pi:\Gcal \to S$ be a $\Gm$-gerbe, corresponding to a torsion class $[\Gcal]$ in the cohomological Brauer group ${\rm Br}'(S)$. Then we have the following exact sequence
    \[ {\rm H}^0(S,\underline{\ZZ}) \longrightarrow {\rm Br}'(S) \longrightarrow {\rm Br}'(\Gcal) \longrightarrow 0\]
    where the first map maps $1$ to $[\Gcal]$.
\end{thm}
Assuming that the gerbe $\Gcal$ is $q$-torsion for $q$ coprime with ${\rm char}(\bfk)$, the result above easily generalizes to the case when $S=\left[ X/G\right]$ is a quotient stack with $X$ a quasi-projective scheme and the action of $G$ is linearized.

Indeed, with these hypotheses there exists a $G$-representation $V$ and $G$-invariant open subset $U\subset V$ whose complement has codimension $>2$ on which $G$ acts freely, such that $X\times^G U:=X\times U/G$ is a scheme \cite{EG}*{Proposition 23}. Then for any $\ell$ we have the equality ${\rm Br}'(X\times^G U)_{\ell}={\rm Br}'(\left[ X/G\right])_{\ell}$.

We have a natural morphism $f:X\times^G U \to [X/G]$ which realizes $X\times^G U$ as an open subset of a vector bundle over $[X/G]$.
Given a $\Gm$-gerbe $\Gcal\to [X/G]$, we obtain a $\Gm$-gerbe $f^*\Gcal \to X\times^G U$. Applying Shin's theorem, we obtain an exact sequence 
\[ {\rm H}^0(X\times^G U,\underline{\ZZ}) \longrightarrow {\rm Br}'(X\times^G U) \longrightarrow {\rm Br}'(f^*\Gcal) \longrightarrow 0\]
where the left arrow maps $1$ to $f^*[\Gcal]$. Now if $\Gcal$ is $q$-torsion the sequence above can be extended to a commutative diagram
\[
    \begin{tikzcd}
    {\rm H}^0(X\times^G U,\underline{\ZZ}) \ar[r] & {\rm Br}'(X\times^G U)_\ell \ar[r] & {\rm Br}'(f^*\Gcal)_\ell \ar[r] & 0 \\
    {\rm H}^0([X/G],\underline{\ZZ}) \ar[r] \ar[u, "\simeq"] & {\rm Br}'([X/G])_\ell \ar[r] \ar[u, "\simeq"] & {\rm Br}'(\Gcal)_\ell \ar[u, "\simeq"] \ar[r] & 0
    \end{tikzcd}
\]
where the bottom left arrow sends $1$ to $[\Gcal]$, and the vertical arrows are the pullbacks along $f$, which by construction are all isomorphisms. Then from the exactness of the top sequence we immediately deduce the exactness of the bottom sequence.

Before getting to the proof the main theorem, we need an intermediate result. Recall that given a ${\rm PGL}_n$-torsor $P \to S$ we have
\[{\rm H}^1(S,{\rm PGL}_n)\ni (P\to S) \mapsto (\left[P/{\rm GL}_n\right]\to S) \in {\rm H}^2(S,\Gm).\] 
On the other hand the ${\rm PGL}_n$-torsor $P$  corresponds to the Brauer-Severi variety $(P\times_S \PP^n)/{\rm PGL}_n$. In particular the universal Brauer-Severi variety $\left[\PP^n/{\rm PGL}_n\right]$ corresponds to the ${\rm PGL}_n$-torsor $\Spec(\bfk) \to {\rm BPGL}_n$ which in turn maps to the $\Gm$-gerbe ${\rm BGL}_n \to {\rm BPGL}_n$. 
\begin{lm}\label{lm:Br PGL}
    Suppose that ${\rm char}(\bfk)\neq 3$ and that $3 \mid \ell$. Then
    \[{\rm Br}'(\Brm\PGLt)_{\ell} \simeq {\rm Br}(\bfk)_{\ell}\oplus \ZZ/3\ZZ\cdot \gamma
    \]
    where $\gamma$ is the class of the $\Gm$-gerbe $\Brm\GLt \to \Brm\PGLt$ or, equivalently, the class of the universal Severi-Brauer variety over $\Brm\PGLt$
\end{lm}
\begin{proof}
    First observe that the class of the $\Gm$-gerbe $\Brm\GLt \to \Brm\PGLt$ is of $3$-torsion as it corresponds to the universal ${\rm PGL}_3$-torsor $\Spec(\bfk) \to \Brm\PGLt$ which has index $3$. Therefore, applying Shin's theorem we get the exact sequence
    \[ \ZZ/3\ZZ \longrightarrow {\rm Br}'(\Brm\PGLt)_{\ell} \longrightarrow {\rm Br}'(\Brm\GLt)_{\ell}={\rm Br}'(\bfk)_{\ell}. \]
    We observe that the first map is injective, as $\Brm\GLt\to\Brm\PGLt$ is a non-trivial gerbe, and the last map has a retraction as ${\rm BPGL}_3$ has a $\bfk$-rational point. This concludes the proof.
\end{proof}
We are now ready to prove the main theorem of this section. Recall \cites{EHKV,KV} that for a smooth, generically tame and separated Deligne-Mumford stack whose coarse moduli space is quasi-projective the Brauer group and cohomological Brauer group coincide. As these conditions are always verified for $\Xcal_d$ we have ${\rm Br}'(\Xcal_d)={\rm Br}(\Xcal_d)$.
\begin{thm}\label{thm:brauer Xd}
Let $\bfk$ be an algebraically closed field whose characteristic is not $2,3$ and does not divide $d$, and let $\Xcal_d$ be the moduli stack of smooth plane curves of degree $d$ over $\bfk$. Then:
\begin{itemize}
    \item if ${\rm char}(\bfk)=0$ then ${\rm Br}(\Xcal_d)\simeq \ZZ/\langle d,6 \rangle$.
    \item if ${\rm char}(\bfk)=p>0$ then ${\rm Br}(\Xcal_d)\simeq \ZZ/\langle d,6 \rangle \oplus B'_{p,d}$, where $B'_{p,d}$ is of $p$-primary torsion.
\end{itemize}
\end{thm}
\begin{proof}
    Observe that we have a cartesian diagram
    \[
    \begin{tikzcd}
        \Xcal_d^{\fr} \ar[r, ] \ar[d, "f"] & \Brm\GLt \ar[d, ] \\
        \Xcal_d \ar[r, "g"] & \Brm\PGLt
    \end{tikzcd}
    \]
    where the horizontal maps come from the presentations of $\Xcal_d$ and $\Xcal_d^{\fr}$ as quotient stacks. From this we deduce that $\Xcal_d^{\fr} \to \Xcal_d$ is a $\Gm$-gerbe whose class in ${\rm Br}'(\Xcal_d)$ is of $3$-torsion. By applying Shin's theorem and \Cref{lm:Br PGL}, we obtain a commutative diagram
    \begin{equation}\label{eq:diag}
    \begin{tikzcd}
    &\ZZ/3\ZZ \ar[r] & {\rm Br}'(\Xcal_d)_\ell \ar[r] & {\rm Br}'(\Xcal_d^{\fr})_\ell \ar[r] & 0 \\
    0\ar[r] & \ZZ/3\ZZ \ar[r] \ar[u, "g^*"] & {\rm Br}'(\Brm\PGLt)_\ell \ar[r] \ar[u, "g^*"] & {\rm Br}'(\bfk)_\ell \ar[u, "g^*"] \ar[r] & 0.
    \end{tikzcd}
    \end{equation}
    Suppose that $3\mid d$. Then $W_d$ is a $\PGLt$-representation, where the action is defined as
    \[ [A] \cdot F(x_1,x_2,x_3) = \det(A)^{\frac{d}{3}}F(A^{-1}(x_1,x_2,x_3)).\]
    This implies that $g\colon\Xcal_d\to\Brm\PGLt$ is a composition of an open embedding and a projective bundle, hence the induced pullback homomorphism $g^*\colon{\rm Br}'(\Brm\PGLt) \to {\rm Br}'(\Xcal_d)$ is injective. From this it follows that the top left morphism in the sequence above is injective, proving ${\rm Br}'(\Xcal_d)_\ell \simeq {\rm Br}'(\Xcal_d^\fr) \times \ZZ/3\ZZ$. 

    Suppose instead that $3 \nmid d$. Then the pullback of the universal Severi-Brauer variety to $\Xcal_d$ is a projective bundle (hence its class is trivial in the Brauer group): indeed, denoting this Severi-Brauer variety as $\pi\colon\Pcal\to\Xcal_d$, we have that over $\Pcal$ there are two well defined line bundles, namely the relative dualizing line bundle $\omega_{\Pcal/\Xcal_d}$ and $\Ocal(\Ccal)$, where $\Ccal\subset \Pcal$ is the universal smooth plane curve.

    Then there exist integers $a$ and $b$ such that $\Lcal:=\Ocal(a\Ccal)\otimes(\omega_{\Pcal/\Xcal_d})^{\otimes b}$ has degree $1$ on every fiber of $\pi\colon\Pcal\to \Xcal_d$, which implies that $\Pcal\simeq \PP(\pi_*\Lcal)$. As the class in the Brauer group of $\Pcal$ coincides with the one of the gerbe $\Xcal_d^{\fr}\to \Xcal_d$, we conclude that the top left arrow in \eqref{eq:diag} is zero, thus proving that ${\rm Br}'(\Xcal_d)\simeq {\rm Br}'(\Xcal_d^\fr)$.
\end{proof}

\begin{rmk}
If $d$ is odd we can again describe the Brauer group of $\Xcal_d$ without assuming that $\bfk$ is algebraically closed:
\begin{enumerate}
    \item if ${\rm char}(\bfk)=0$ we have \[{\rm Br}(\Xcal_d^\fr) \simeq {\rm Br}(\bfk) \oplus {\rm H}^1(\bfk,\ZZ/3^{i_d}(d-1)^2\ZZ) \oplus \ZZ/\langle 3, d\rangle.\]
    \item if ${\rm char}(\bfk)=p>0$ we have 
    \[{\rm Br}(\Xcal_d^\fr)_\ell \simeq {\rm Br}(\bfk)_\ell \oplus {\rm H}^1(\bfk,\ZZ/3^{i_d}(d-1)^2\ZZ)_\ell\oplus \ZZ/\langle 3, d\rangle_\ell \]
\end{enumerate}
To see the splitting of exact sequence note that the torsion in the Picard group of $\Xcal_d$ is the same as the torsion in the Picard group of $\Xcal_d^{\fr}$, which can be seen using \cite{Shi}*{Eq. 2.9.2} and the fact that the character group of $\Gm$ is $\ZZ$. This shows that $\H^1(\Xcal_d^{\fr},\mu_{3^{i_d}(d-1)^2})$ contains an element $\tilde{\beta}_1$ which maps to the element $\beta_1 \in \H^1(\Xcal_d^{\fr},\mu_{3^{i_d}(d-1)^2})$ generating the corresponding component in ${\rm Br}(\Xcal_d^{\fr})$. Then the map
\[{\rm Br}(\Xcal_d^{\fr})={\rm Br}(\bfk)\oplus{\rm H}^1(\bfk,\ZZ/3^{i_d}(d-1)^2\ZZ)\to {\rm Br}(\Xcal_d), \quad (a,x) \mapsto a + \tilde{\beta}_1 \cdot x\]
is a splitting of ${\rm Br}(\Xcal_d) \to {\rm Br}(\Xcal_d^{\fr})$.
\end{rmk}

\section{The Brauer group of $\Mcal_3$}
We want to compute the Brauer group of $\Mcal_3$, the moduli stack of smooth curves of genus $3$, over a ground field $\bfk$ (not necessarily algebraically closed). We do this by computing the low-degree cohomological invariants of $\Mcal_3$ (\Cref{thm:inv M3}). Interestingly, the generator for the (normalized) low-degree cohomological invariants of $\Mcal_3$ comes from the theory of theta-characteristics, and its non-triviality is proved via an argument involving the $28$ bitangents of a smooth quartic and the $27$ lines over a smooth cubic surface (see \Cref{subsection:cubic}).
\subsection{Genus three curves, their canonical model and theta-characteristics}\label{subsec:genusthree}
Recall that the canonical model of a smooth, non-hyperelliptic curve of genus three over a field is a plane quartic curve in $\PP^2_{\bfk}$. Therefore, we have the following guiding principle: every intrinsic geometric property of a non-hyperelliptic curve of genus three has an interpretation in terms of plane quartic curves, and viceversa.

Let $\Mcal_3$ be the moduli stack of smooth genus three curves, which is a smooth, connected Deligne--Mumford stack, and let $\Hcal_3$ be the divisor of hyperelliptic curves. A first instance of the principle above is then the following.
\begin{prop}\label{prop:M3-H3 iso X4}
    There is an isomorphism $\Mcal_3\smallsetminus\Hcal_3 \simeq \Xcal_4$.
\end{prop}
\begin{proof}
    Given an object $(C\hookrightarrow P \to S)$ of $\Xcal_4$, by forgetting the inclusion we get a smooth, non-hyperelliptic genus three curve $(C\to S)$. This defines a morphism $f\colon\Xcal_4\to\Mcal_3\smallsetminus\Hcal_3$.
    
    Viceversa, given a smooth, non-hyperelliptic, genus three curve $\pi\colon C\to S$, we have the canonical, surjective morphism $\pi^*\pi_*\omega_{\pi}\to \omega_\pi$, where $\omega_\pi$ is the relative dualizing sheaf. This induces a morphism $C\hookrightarrow \PP(\pi_*\omega_\pi)\to S$, which is a closed embedding and realizes each geometric fiber as a smooth plane curve of degree four. 
    
    In this way we have a morphism $g\colon\Mcal_3\smallsetminus\Hcal_3\to \Xcal_4$ which constitutes an inverse to the previous map. Indeed, it is straightforward to check that $g\circ f={\rm id}$. On the other hand, given an object $C\overset{\iota}{\hookrightarrow} P \overset{\rho}{\to} S$, observe that $L:=\omega_{P/S}\otimes \Ocal(C)$ has degree one on every fiber of $P\to S$, hence we have a canonical isomorphism $P\simeq \PP(\rho_*L)$.
    
    Observe that there is an isomorphism $L|_C\simeq \omega_{C/S}\otimes \pi^*M$ for some line bundle $M$ on the base. This implies that $\PP(\pi_*\omega_{C/S})$ is canonically isomorphic to $\PP(\pi_*(L|_C))$. The morphism $\rho_*L \to \rho_*(\iota_*(L|_C))\simeq \pi_*(L|_C)$ is an isomorphism, because both $\rho_*L(-C)$ and $R^1\rho_*(L(-C))$ are zero; this implies that $\PP(\pi_*(L|_C))\simeq \PP(\rho_*L)\simeq P$, hence $f\circ g={\rm id}$.
\end{proof}
We will also need the following.
\begin{prop}[\cite{DilPicPositive}*{Proposition 3.1.3}]\label{prop:presentation M3 minus H3}
    Let $W_4$ be the vector space of quartic forms in three variables, endowed with the $\GLt$-action $A\cdot f(x_0,x_1,x_2):=\det(A)f(A^{-1}(x_0,x_1,x_2))$, and let $\Delta \subset W_4$ be the divisor of singular quartic forms. Then we have an isomorphism of Deligne--Mumford stacks
    \[ \Mcal_3\smallsetminus\Hcal_3 \simeq [ (W_4\smallsetminus \Delta)/\GLt]. \]
\end{prop}
A second instance of the principle above is provided by theta-characteristics. Recall that a theta-characteristic on a smooth curve $C$ over a field $\bfk$ is an isomorphism class of a line bundle $\theta$ such that $\theta^{\otimes 2}\simeq \omega_C$, where $\omega_C$ denotes the canonical bundle of the curve. The set of theta-characteristics in the Picard variety ${\rm Pic}^{g-1}(C)$ forms a torsor under the additive group ${\rm Jac}(C)_2$ of $2$-torsion points in the Jacobian of $C$, i.e. the connected component of the identity of ${\rm Pic}(C)$. In particular, a smooth curve of genus $g$ has exactly $2^{2g}$ theta-characteristics.

Theta-characteristics can be further subdivided in even or odd, depending on the parity of $h^0(C,\theta)$. A smooth curve $C$ of genus $g$ has $2^{g-1}(2^{g}-1)$ odd theta-characteristics and $2^{g-1}(2^{g}+1)$ even theta-characteristics.

If $C$ is a smooth curve of genus three, we deduce that there are exactly $28$ odd theta-characteristics, and for each of them we must have $h^0(C,\theta)=1$, because $h^0(C,\theta)\leq \deg(\theta)=2$. On the other hand, global sections of $\theta$ correspond to effective divisor of degree two on $C$, hence there are exactly $28$ distinct pairs of (non necessarily distinct) points $p+q$ such that $\Ocal(2p+2q)\simeq \omega_{C}$. Given a canonical embedding $C\subset \PP^2$, we have that $\omega_{C}\simeq \Ocal_{\PP^2}(1)|_C$, hence the previous statement can be reformulated as follows.
\begin{prop}\label{prop:bitang}
    Given a canonically embedded, smooth genus three curve $C\subset \PP^2_K$ over a field $K$, there is a bijection between the set of odd theta-characteristics of $C$ and the set of lines in the plane that are bitangent to $C$. In particular, every canonically embedded smooth curve of genus three has exactly $28$ bitangents.
\end{prop}
The notion of theta-characteristics can be extended to families of curves: given $\pi\colon C\to S$ a smooth curve of genus three, a theta-characteristic consists of a line bundle $\theta$ on $C$ and an isomorphism $\theta^{\otimes 2} \overset{\sim}{\to} \omega_{\pi}$. It is known (see for instance \cite{MumSpin}) that the parity of the rank function of the sheaf $\pi_*\theta$ is constant, thus the definition of odd or even theta-function extends to families of curves as well. The datum $$(\pi\colon C\to S, \theta, \alpha\colon \theta^{\otimes 2}\overset{\sim}{\to}\omega_\pi)$$ with $\theta$ an odd (respectively even) theta-characteristic is called an \emph{odd spin curve} (respectively an \emph{even spin curve}).
\begin{prop}
    There exists a smooth Deligne--Mumford stack $\Scal^{-}_3$ whose objects are odd spin curves. The morphism $\Scal^{-}_3 \to \Mcal_3$ that forgets the theta-characteristic is a representable, finite and \'{e}tale morphism of degree $28$.
\end{prop}

\subsection{Extending invariants from moduli of plane curves}
In this subsection we will build on the results of the previous section to obtain a partial computation of the cohomological invariants of $\Mcal_3$ and $\Mcal_3 \smallsetminus \Hcal_3$. First we show that the low degree invariants of $\Xcal_{d}^\fr$ and $\Mcal_3\smallsetminus\Hcal_3$ coincide.
\begin{prop}\label{prop:nonHypiso}
The $\Gm$-torsor $\Mcal_3 \smallsetminus \Hcal_3 \xrightarrow{\pi} \Xcal_{4}^\fr$ induces an isomorphism on cohomological invariants.
\end{prop}
\begin{proof}
First note that the pullback $\pi^*$ is injective because a $\Gm$-torsor is always a smooth-Nisnevich covering. We know by \Cref{prop:M3-H3 iso X4} that the composition $\Mcal_3 \smallsetminus\Hcal_3 \to \Xcal_4$ is an isomorphism, which implies that the composition
\[
\Inv(\Xcal_4,\M) \to \Inv(\Xcal_4^{\rm fr},\M) \to \Inv(\Mcal_3 \smallsetminus \Hcal_3, \M)
\]
is an isomorphism (note that we can pullback cohomological invariants through non representable maps), showing that $\pi^*$ must be surjective as well.
\end{proof}

The stack $\Hcal_3$ of hyperelliptic curves of genus three is a smooth Deligne--Mumford stack, whose generic stabilizer is isomorphic to $\mu_2$, and it forms a Cartier divisor in $\Mcal_3$. The open embedding $\Mcal_3\smallsetminus\Hcal_3 \hookrightarrow \Mcal_3$ induces an inclusion at the level of cohomological invariants. We will show that this inclusion is surjective on $2$-torsion invariants.

\begin{lm}
Let $M_3^{\sm}, H_3^{\sm}$ respectively be the smooth locus of the \emph{moduli space} of smooth genus three curves and hyperelliptic smooth genus three curves. Then:
\begin{enumerate}
    \item $M_3^{\sm}$ is the locus of curves which either have no nontrivial automorphisms or whose only nontrivial automorphism is the hyperelliptic involution.
    \item $H_3^{\sm}$ is the Cartier divisor in $M_3^{sm}$ of hyperelliptic curves whose automorphism group is exactly $\ZZ/2\ZZ$.
    \item $M_3 \smallsetminus M_3^{\sm}$ has codimension $\geq$ 2 in $M_3$.
\end{enumerate}
\end{lm}
\begin{proof}
 The first statement is proven in \cite{Oort}. The same reasoning, restricted to $H_3$ and its image, proves the second statement, and the third is \cite{Oort}*{Lm. 2}.
\end{proof}

Set $\Mcal_3^{\sm}:=\Mcal_3\times_{M_3} M_3^{\sm}$. Observe that as the complement of $\Mcal_3^{\sm}$ in $\Mcal_3$ has codimension $\geq 2$ we have
\[ \Inv(\Mcal_3, \M) \simeq \Inv(\Mcal_3^{\sm}, \M). \]

\begin{prop}\label{prop:ismomroot}
Let $M^{\sm}_{3, \textnormal{root}}$ be the order two root stack of $M_3^{\sm}$ along the Cartier divisor $H_3^{\sm}$. We have an isomorphism:
\[\Mcal_3^{\sm} \simeq M^{\sm}_{3, \textnormal{root}}.\]
\end{prop}
\begin{proof}
    The coarse moduli space morphism $\Mcal_3^{\sm} \to M_3^{\sm}$ ramifies with order two over $H_3^{\sm}$, thus \cite{GS}*{Theorem 1} we have a factorization
\[ \Mcal_3^{\sm} \longrightarrow M^{\sm}_{3, \textnormal{root}} \longrightarrow M_3^{\sm}. \]
 The first morphism in the factorization above is representable by schemes: it is representable by algebraic spaces because it induces an isomorphism on automorphism groups, and it is quasi finite, so by \cite{GMB}*{Th\'eor\'eme A.2} it is representable by schemes. Moreover it is by construction universally closed, generically an isomorphism and surjective. Applying Zariski's main theorem, we deduce that $\Mcal_3^{\sm} \to M^{\sm}_{3, \textnormal{root}}$ is an isomorphism.
\end{proof}
 
As a consequence, we have an isomorphism of cohomological invariants
\[\Inv(\Mcal_3, \M) \simeq \Inv(\Mcal_3^{\sm}, \M) \simeq \Inv(M^{\sm}_{3, \textnormal{root}}, \M) \]
and we get

\begin{prop}\label{prop:extensionM3}
We have an equality 
\[ \Inv(M^{\sm}_3\smallsetminus H_3^{\sm}, \M) = \Inv(\Mcal_3\smallsetminus\Hcal_3, \M)\] 
and the subgroup 
\[\Inv(\Mcal_3,\M) \subseteq \Inv(\Mcal_3 \smallsetminus \Hcal_3,\M)\]
is composed of those invariants whose ramification at $H_3^{\sm}$ is of $2$-torsion.

In particular this implies that
\[\Inv(\Mcal_3\smallsetminus\Hcal_3, \M)_2 \simeq \Inv(\Mcal_3, \M)_2. \]
\end{prop}
\begin{proof}
Applying the formula for the invariants of a root stack \cite{DilPirRS}*{Thm. 1.5} we have that $\Inv(M^{\sm}_{3, \textnormal{root}}, \M)$ is equal to the subset of $\Inv(M^{\sm}_3\smallsetminus H_3^{\sm}, \M)$ whose elements have $2$-torsion ramification at $H_3^{\sm}$.

On the other hand, we have an isomorphism $M_3^{\sm}\smallsetminus H_3^{\sm} \simeq \Mcal_3^{\sm}\smallsetminus\Hcal_3^{\sm}$, and the complement of the latter in $\Mcal_3\smallsetminus \Hcal_3$ has codimension $\geq 2$. We deduce that
\[ \Inv(M^{\sm}_3\smallsetminus H_3^{\sm}, \M) = \Inv(\Mcal_3\smallsetminus\Hcal_3, \M) .\]   
The last statement follows trivially as the ramification of a $2$-torsion invariant is always of $2$-torsion.
\end{proof}

This shows that the subgroup $N\left[2 \right] \subset \Inv(\Mcal_3 \smallsetminus \Hcal_3, \M)$, which is $2$-torsion, belongs to the cohomological invariants of $\Mcal_3$. All that's left to do is understanding if the same happens for the $9$-torsion subgroup $\beta_1 \cdot \M^{\bullet}(\bfk)_9$.

\begin{cor}\label{cor:invM3}
We have 
\[
\Inv(\Mcal_3,\M^{\leq 2}) = \Inv(\bfk, \M^{\leq 2}) \oplus N\left[ 2\right]
\]
where $N\left[2\right]$ restricts on $\Mcal_3 \smallsetminus \Hcal_3$ to the pullback of the corresponding summand in $\Inv(\Xcal_{4}^{\fr},\M)$.
\end{cor}
\begin{proof}
We want to see which elements of 
\[\Inv(\Mcal_3 \smallsetminus \Hcal_3, \M^{\leq 2})=\M^{\leq 2}(\bfk) \oplus \beta_1 \cdot \M^{\bullet}(\bfk)_9 \oplus N\left[2\right]\]
extend to $\Mcal_3$. First note that obviously all invariants coming from the base field extend, and by \Cref{prop:extensionM3} the $2$-torsion component, which includes $N\left[2\right]$, extends to $\Mcal_3$. Now, by \cite{DilPicPositive}*{Theorem 3.1.5} we know that ${\rm Pic}(\Mcal_3)=\ZZ$, and it's well known that $\mathcal{O}^*(\Mcal_3)=\bfk^*$. Finally, recall that by \cite{DilPirBr}*{Lemma 2.18} we have ${\rm Inv}^1(\Mcal_3,\K_\ell)={\rm H}^1(\Mcal,\mu_\ell)$. Then the Kummer exact sequence for $\ell = 9$ reads
\[
\bfk^*/(\bfk^*)^9 \to {\rm Inv}^1(\Mcal_3,\K_9) \to {\rm Pic}(\Mcal_3)_9 = 0
\]
the component on the left is just $\K^1_9(\bfk)$ so we conclude that there are no elements in ${\rm Inv}^1(\Mcal_3,\K_9)$ outside those coming from the base, which implies that the ramification of $\beta_1$ on $\Hcal_3$, which belongs to ${\rm A}^0(\Hcal_3,{\rm H}^0_9)=\ZZ/9\ZZ$, must be coprime to $3$. But then the ramification of $\beta_1 \cdot x$ must be nonzero for any $x \in \M^{\bullet}(\bfk)_9$, proving our claim.
\end{proof}

\subsection{\`Etale algebras and Stiefel-Whitney classes}

To explicitly construct a generator for the factor $N[2]$ of the ring of cohomological invariants of $\Mcal_3$ we will make use of some results in the theory of classical cohomological invariants, to the authors' knowledge due to Serre. 

An \emph{\'etale algebra} $E$ over $F$ is a finite $F$-algebra which is isomorphic to $(F')^n$ for some $n$ after passing to some finite extension $F'/F$, i.e. $\Spec(E) \to \Spec(F)$ is a finite \'etale morphism. Passing to the group scheme $T_E(R) = {\rm Aut}_R(E \otimes R)$ we obtain a ${\rm S}_n$-torsor over $F$. The category of finite \'etale algebras of rank $n$ over $\bfk$ forms a Deligne-Mumford stack $\textnormal{\'Et}_n$ isomorphic to ${\rm BS}_n$, and any finite \'etale morphism $\mathcal{S} \to \Xcal$ of constant degree $n$ induces a morphism $\Xcal \to \textnormal{\'Et}_n$, meaning we can pull the cohomological invariants of $\textnormal{\'Et}_n$ back to $\Xcal$.

The cohomological invariants of $\textnormal{\'Et}_n$ are fully understood \cite{GMS}*{Ch. 7}, and can be described as follows:

Given $E/F \in \textnormal{\'Et}_n(F)$, we can define a scalar product on $E$, seen as a $F$-vector space, by $\langle x, y \rangle_{\rm tr} = {\rm tr}(\cdot xy)$, that is the couple $x,y$ is sent to the trace of the multiplication by $xy$ seen as an endomorphism of $E$. Diagonalizing this product yields an $n$-uple of coefficients $(d_1, \ldots, d_n) \in F^n$; obviously picking two different diagonal representations will yield different coefficients: it's immediate that an invariant of $E$ must be left unchanged when multiplying coefficients by a square or permuting them. We define the $i$-th Stiefel-Whitney class of $E$ as
\[
\alpha^{\rm SW}_i(E) = \sigma_i(d_1,\ldots, d_n) \in \K^i_{2}(F)
\]
where $\sigma_0=1 \in \K^0_2(F)$ and $\sigma_i$ is the $i$-th elementary symmetric polynomial in $d_1, \ldots, d_n$ seen as element in $\K_2^i(F)=\H^i(F,\ZZ/2\ZZ(i))$, e.g. 
\[\sigma_2(a,b,c)=\lbrace a , b \rbrace + \lbrace a , c \rbrace + \lbrace b , c \rbrace \in \K^2_2(\bfk(a,b,c)).\]
We have
\[
\Inv(\textnormal{\'Et}_n,\M)=\M^{\bullet}(\bfk) \oplus \alpha^{\rm SW}_1 \cdot \M^{\bullet}(\bfk) \oplus \ldots \oplus \alpha^{\rm SW}_{\lfloor n/2\rfloor} \cdot \M^{\bullet}(\bfk). 
\]
Now we will define new generators for the cohomological invariants of $\textnormal{\'Et}_n$, the Galois-Stiefel-Whitney classes, which have better properties. Set \cite{GMS}*{Thm.25.10}
\[
\alpha_i=
\begin{cases}
\alpha_i^{\rm SW} \quad & i \;\textnormal{odd}\\
\alpha_i^{\rm SW} + \lbrace 2 \rbrace \alpha^{\rm SW}_{i-1}\quad & i\; \textnormal{even}
\end{cases}
\]

Moreover, define the total Galois-Stiefel-Whitney class of $E$ as $\alpha_{\rm tot}(E)=\sum_i \alpha_i(E)$. Then the following properties hold:
\begin{enumerate}
    \item $\Inv(\textnormal{\'Et}_n,\M)=\M^{\bullet}(\bfk) \oplus \alpha_1 \cdot \M^{\bullet}(\bfk) \oplus \ldots \oplus \alpha_{\lfloor n/2\rfloor} \cdot \M^{\bullet}(\bfk). $
    \item $\alpha_i(E)=0$ if $i$ is greater than $\lfloor n/2 \rfloor$.
    \item $\alpha_{\rm tot}(F^n/F)=1$.
    \item $\alpha_{\rm tot}(F(\sqrt{a})/F)=1+\lbrace a \rbrace$.
    \item $\alpha_{\rm tot}(E \times E')=\alpha_{\rm tot}(E)\cdot \alpha_{\rm tot}(E').$
\end{enumerate}

The following examples will be helpful later on.

\begin{ex}\label{ex:alphatot}
Consider the field $F=\bfk(a,b)$, and the \'etale algebra over $F$
\[
E= F(\sqrt{a}) \times F(\sqrt{b}) \times F(\sqrt{ab}).
\]
Then we have 
\[\alpha_{\rm tot}(E) = \alpha_{\rm tot}(F(\sqrt{a})) \cdot \alpha_{\rm tot}(F(\sqrt{b})) \cdot \alpha_{\rm tot}(F(\sqrt{c}))=(1+ \lbrace a \rbrace)(1+ \lbrace b \rbrace)(1+ \lbrace ab \rbrace)\]
\[=1+\lbrace a \rbrace + \lbrace b \rbrace + \lbrace ab \rbrace + \lbrace a, b \rbrace + \lbrace a, ab \rbrace + \lbrace b, ab \rbrace + \lbrace a, b, ab \rbrace =\]
\[1 + 2\lbrace ab \rbrace + 3\lbrace a, b \rbrace + \lbrace - 1, a \rbrace + \lbrace -1, b \rbrace - 2\lbrace -1, a, b \rbrace = \]
\[=1 + \lbrace a, b \rbrace + \lbrace -1, a \rbrace +\lbrace -1, b \rbrace = 1 +\lbrace a , b \rbrace + \lbrace -1, ab \rbrace.\]
\end{ex}

\begin{ex}\label{ex:alphaK}
Consider again  $F=\bfk(a,b)$ and let $K=F(\sqrt{a},\sqrt{b})$, seen as an \'etale algebra over $F$. We know that $\alpha_i(K)=0$ for $i >2$, so we are interested in computing $\alpha_1(K)$ and $\alpha_2(K)$. Consider the trace form on $K$. By direct computation we obtain 
\[
{\rm tr}(\cdot \sqrt{a})={\rm tr}(\cdot \sqrt{b})={\rm tr}(\cdot \sqrt{ab})=0, {\rm tr}(\cdot 1)=4, {\rm tr}(\cdot a)=4a, {\rm tr}(\cdot b)=4b, {\rm tr}(\cdot ab)=4ab
\]
which shows that the matrix representing the trace form is
\[
\begin{pmatrix} 4 & 0 & 0 & 0 \\ 0 & 4a & 0 & 0\\ 0 & 0 & 4b & 0 \\ 0 & 0 & 0 & 4ab \end{pmatrix}
\]
which shows that 
\[\alpha_1(K)=\alpha^{\rm SW}_1(K)=\lbrace a \rbrace+\lbrace b \rbrace+\lbrace ab \rbrace=0\]
and in degree two we have
\[\alpha_2(K)=\alpha^{\rm SW}_2(K)+\lbrace 2\rbrace\alpha^{\rm SW}_1(K)=\lbrace a , b \rbrace+\lbrace a , ab \rbrace+\lbrace b , ab \rbrace= \lbrace a , b \rbrace + \lbrace -1, ab \rbrace.\]
\end{ex}

\subsection{An explicit generator: bitangents and lines on cubic surfaces}\label{subsection:cubic}

In this section we will produce an explicit \'etale algebra $S$ over $\Mcal_3$ and $\Xcal_{4}^{\fr}$ whose second Galois-Stiefel-Whitney class is non-trivial, or more specifically such that the map ${\rm M}^{\bullet}(\bfk)_2 \to \Inv(\Xcal_{4}^{\fr}, \M)$ given by $x \mapsto \alpha_2(S) \cdot x$ is injective for any cycle module $M$. 

Recall from \Cref{subsec:genusthree} that there is a finite étale morphism $\mathcal{S}_3^- \to \Mcal_3$ of degree $28$, which is given by the functor $(C,\theta)\mapsto C$ that forgets the odd theta-characteristic $\theta$. 

There are two maps 
\[\Mcal_3 \smallsetminus \Hcal_3 \to \Xcal_4, \quad \Xcal_4 \to \Mcal_3 \smallsetminus \Hcal_3\]
the first being the composition of $\Mcal_3 \smallsetminus \Hcal_3 \to \Xcal_4^{\fr}$ and $\Xcal_4^{\fr} \to \Xcal_4$ and the second given by the forgetful map $(C\to P \to S) \mapsto (C \to S)$. It's immediate to check that their composition is the identity on $\Mcal_3 \smallsetminus \Hcal_3$. This shows that the restriction of $\mathcal{S}_3^-$ to $\Mcal_3 \smallsetminus \Hcal_3$ induces a corresponding \'etale covering of $\Xcal_4^{\fr}$ and $\Xcal_4$, and everything we do can be applied to either of these. By \Cref{prop:bitang}, given a non-hyperelliptic curve $C/F$ and an embedding $C \to \PP^2_F$ the fiber $\mathcal{S}^-_3(C)$ parametrizes bitangent lines to $C$.  

Name $S$ the corresponding \'etale algebra on $\Mcal_3$ and $\Xcal^{\fr}_4$; We will show:

\begin{prop}\label{prop:explicit}
Let $N\left[2 \right]$ be the subgroup of $\Inv(\Xcal^{\fr}_4,\M)$ and $\Inv(\Mcal_3, \M)$ in Proposition \ref{prop: InvX4} and Corollary \ref{cor:invM3}. We have
  \[  N\left[2\right] = \M^{\bullet}(\bfk)_2\left[ 2\right] = \alpha_2(S)\cdot \M^{\bullet}(\bfk)_2\]
where $S$ is the \'etale algebra  coming from the finite \'etale morphism of degree $28$
\[
\mathcal{S}_3^- \to \Mcal_3.
\]
and the corresponding morphism on $\Xcal^{\fr}_4$.
\end{prop}
This will complete the computation of low degree cohomological invariants of $\Mcal_3$ as well as $\Xcal^{\fr}_4$ and $\Mcal_3 \smallsetminus \Hcal_3$.

We recall now a very well known construction of a smooth cubic surface as a two-sheeted covering of $\PP^2$, branched along a smooth quartic (see for example \cite{Harris}).

\begin{prop}
Assume ${\rm char}(F) \neq 2$, and consider a smooth cubic surface $X \subset \PP^3_F$, an $F$-valued point $P \in X$ which does not lie on any line of $X$. Let $\widetilde{X}={\rm Bl}_P X$ be the blow-up of $X$ at $P$. Then there is a finite morphism of degree two 
\[
f:\widetilde{X} \to \PP^2_F
\]
whose ramification locus is a smooth quartic $B_X \subset \PP^2_F$. Moreover the bitangent lines to $B_X$ correspond to the lines on $X$ plus the exceptional divisor.
\end{prop}
\begin{proof}
Let $H$ be a hyperplane in $\PP^2_F$, not containing $P$.
We can define a map $X\smallsetminus P \to H$ by sending $Q\in X$ to the intersection of $H$ and the line through $P$ and $Q$. By taking the blow-up of $X$ at $P$ we can extend this map to a projection
\[f:\widetilde{X}={\rm Bl}_P X \to H.\]
This is a finite morphism of degree two as every line in $\PP^3$ which passes through $P \in X$ intersects $X$ in two more points, counted with multiplicity.

Observe that the canonical divisor $K_{\widetilde{X}}$ is equal to $f^*K_{H}+R_X$, where $R_X$ is the ramification divisor. From this we deduce
\begin{align*}
    f_*(K_{\widetilde{X}})&= \deg(f)\cdot K_H + B_X = -6L + B_X,
\end{align*}
where $B_X$ is the branching divisor and $L$ is the class of a line in $H$. On the other hand, from the formula for the canonical divisor of a blow-up, we have that $K_{\widetilde{X}} = b^*K_X + E$, where $b\colon \widetilde{X}\to X$ is the blow-down morphism and $E$ is the exceptional divisor. 

As $X\subset\PP^3$ is a cubic hypersurface, by the adjunction formula $K_X = -S|_X$, where $S\subset \PP^3$ is a hyperplane. Moreover, from the formula for the fundamental class of a strict transform, we have $b^*(S|_X)=C + E$, where $C$ is the proper transform of $S|_X$. We deduce that 
$K_{\widetilde{X}} = -C$.

Observe that for a generic hyperplane $S\subset\PP^3$, the curve $S|_X$ is a plane curve of degree three in $S\simeq\PP^2$, hence a curve of genus one. On the other hand, the closure of the image of $S \subset \PP^3_F$ through the projection $\PP^3\dashrightarrow H$ is a line $L$, because the projection from a point maps generic hyperplanes to lines. These two facts together imply that the induced morphism $f|_C\colon C \to f(C)=L$ is a 2:1 covering, hence $f_*C= 2L$.

Putting everything together we deduce that $B_X=4L$, hence $f\colon\widetilde{X} \to H$ is a two-sheeted covering whose ramification locus is a smooth plane quartic $B_X$.

Let $D\subset X$ be a line, which by hypothesis does not contain the point $P$. Then its image through $X\smallsetminus P \to \PP^2_F$ is a curve; on the other hand, the projection of $D$ through $\PP^3_F\smallsetminus P \to \PP^2_F$ is either a line or a point, hence it must be a line. If $L\subset \PP^2_F$ is this line, then the restriction of $f|_D\colon D\to L$ has degree one, where with a little abuse of notation we are denoting both $D$ and its strict transform in $\widetilde{X}$ with the same letter.

As $f^{-1}(L)\to L$ is finite of degree two, and $D\subset f^{-1}(L)$ is not contained in the branching locus, we must have $f^{-1}(L)=D\cup C$ for some other curve $C\subset \widetilde{X}$ which is mapped 1:1 onto $L$.

Let $S\subset \PP^3_F$ be the hyperplane obtained as the closure of the preimage of $L$ through the projection $\PP^3_F\smallsetminus P \to \PP^2_F$; then $S\cap X= D \cup \overline{C}$, where $\overline{C}$ is the image of $C$ through the blow-down. 

As $D \cup \overline{C}\subset S$ is a plane curve of degree three, we must have that $\overline{C}$ has degree two, hence it meets $D$ in two points, counted with multiplicity (which must be different from $P$). By construction, the ramification points of $D\cup C \to L$ coincides with $D\cap C$, hence $L$ is a line meeting the branching divisor in two points. We deduce that $L$ is bitangent to $B_X$.

Finally, observe that the exceptional divisor $E$ is also mapped to a line $L$: this follows from the fact that the morphism ${\rm Bl}_P \PP^3_F \to \PP^2_F$ maps the exceptional divisor 1:1 onto $\PP^2_F$, and the map $E\to f(E)$ factors through the embedding of $E$ in the exceptional divisor of ${\rm Bl}_P \PP^3_F$.

Then as before we must have $f^{-1}(L)=E\cup C$ for some curve $C\subset \widetilde{X}$ that maps 1:1 onto $L$. Let $\overline{C}$ be its image in $X$: then $\overline{C}$ is equal to the intersection of $X$ with some hyperplane $S$, hence must have genus one. The only option left is that $\overline{C}$ is singular at $P$, from which we deduce that $C$ meets $E$ in two points (counted with multiplicity). 

As the branching locus of $f^{-1}(L) \to L$ coincides with $f(E\cap C)$, we deduce that $L$ is also a bitangent to the ramification locus.

As the number of bitangents of $B_X$ is $28$, we deduce that the $27$ lines on $X$ are sent $1:1$ to the $27$ bitangents of $B_X$, with the last bitangent being the image of the exceptional divisor.     
\end{proof}

Thus if we pick $B_X \in \Xcal_4(F)$ obtained this way, the fiber $\mathcal{S}^{-}_3(B_X)$ is equal to $\Spec(F) \cup L_X$, where $L_X$ is the Fano scheme of lines on $X$. 

Our goal is then to produce a cubic surface $X \subset \PP^3_{F}$ such that the scheme $L_X$ representing lines on $X$, or more precisely the associated \'etale algebra $E_X/F$ has a second Galois-Stiefel-Whitney class which satisfies $x \mapsto \alpha_2(E_X) \cdot x \neq 0$ for any cycle module $\M$ and $x \in \M^{\bullet}(\bfk)$. 

Let $F$ be the field $\bfk(a,b)$. A well known way to produce a cubic surface in $\PP^3$, see e.g. \cite{Hartshorne}*{Ch. V, Section 4}, is to blow up six points in general position on $\PP^2$; we want to make a more arithmetic version of this construction, taking inspiration from Example \ref{ex:alphatot}. Consider three length two points $P_1, P_2, P_3$ on $\PP^2_F$, defined respectively over $E_1=F(\sqrt{a})$, $E_2 = F(\sqrt{b})$ and $E_3 = F(\sqrt{ab})$. If we write $K=F(\sqrt{a}, \sqrt{b})$ and ${\rm Gal}(K) = \langle \sigma, \tau \rangle$ where 
\[\sigma(\sqrt{a})=-\sqrt{a}, \quad \sigma(\sqrt{b})=\sqrt{b}, \quad \tau(\sqrt{a})=\sqrt{a}, \quad \sigma(\sqrt{b})=-\sqrt{b}\]
this is equivalent to picking six points 
$Q_1, Q'_1, Q_2, Q'_2, Q_3, Q'_3$ on $\PP^2_{K}$ where the same-numbered points are exchanged by the Galois group and are fixed by respectively $\tau, \sigma$ and $\tau \cdot \sigma$. We moreover ask for these points to be in general position, i.e. no three points belong to the same line and the six points are not all on a conic. An example of six such points that works for every base field is
\[
Q_1, Q'_1 = ( \pm\sqrt{a}: 1 : 0), Q_2, Q'_2=(\pm\sqrt{b}: 0: 1), Q_3, Q'_3 = (0 : \pm\sqrt{ab} : 1)
\]
Blowing up these three points yields a cubic surface $X/F$, as can be seen by passing to $K$.

The $27$ lines on a cubic surface obtained this way can be easily described. Over $K$, they are the six exceptional divisors corresponding to the points, the fifteen lines passing through two points each, and the six conics passing through five of the six points. Then to describe the \'etale algebra $E_X/F$ it suffices to look at the action of the Galois group on $E_X \otimes_F K= K^{27}$. We immediately see that:

\begin{enumerate}
    \item The six lines given by the exceptional divisors are fixed by the stabilizer of the corresponding point and thus are defined over the same field.
    \item The three lines passing by conjugate points are fixed by the whole group and thus defined over $F$.
    \item The twelve lines passing by two non-conjugate points are only fixed by the identity and thus are defined over $K$.
    \item The six conics passing through five of six points are stabilized by the stabilizer of the sixth point and thus are defined over the same field.
\end{enumerate}

This shows that 
\[
E_X = F^3 \times E_1^2 \times E_2^2 \times E_3^2 \times K^3
\]
We know that $\alpha_{\rm tot}(\Spec(F))=1$ and $\alpha_{\rm tot}(\Spec(E_i))$ is equal to $1+\lbrace a\rbrace,1+\lbrace b\rbrace$ and $1+\lbrace ab\rbrace$ respectively, which shows
\[
\alpha_{\rm tot}(E_X)=(1+\lbrace a \rbrace)^2(1+\lbrace b \rbrace)^2(1+\lbrace ab \rbrace)^2\alpha_{\rm tot}(\Spec(K))^3
\]
expanding the first three components we get
\[
(1+\lbrace a \rbrace)^2(1+\lbrace b \rbrace)^2(1+\lbrace ab \rbrace)^2=(1+\lbrace -1 , a \rbrace)(1+\lbrace -1 , b \rbrace)(1+\lbrace -1 , ab \rbrace)=
\]
\[
1 + \lbrace -1 , a \rbrace + \lbrace -1 , b \rbrace + \lbrace -1 , a \rbrace + \lbrace -1 , b \rbrace + r = 1 + r
\]
Where every term in $r$ has degree higher than $2$. This shows that the degree two term of $\alpha_{\rm tot}(E_X)$ is the degree two term of $\alpha_{\rm tot}(K)^3$. By example \ref{ex:alphaK} we know that 
\[\alpha_1(K)=0, \quad \alpha_2(K)=\lbrace a, b \rbrace + \lbrace -1, ab \rbrace \]
so
\[
\alpha_{\rm tot}(E_X)=(1+r)\alpha_{\rm tot}(K)^3=(1+r)(1+\lbrace a , b \rbrace + r')^3=1+\lbrace a , b \rbrace + \lbrace -1, ab \rbrace + r''
\]
Where again every term in $r''$ has degree higher than $2$, showing that
\[
\alpha_2(E_X)=\lbrace a , b \rbrace + \lbrace -1, ab \rbrace.
\]
Observe that for any $x \in \M^{\bullet}_2(\bfk)$ by the compatibility of the $\K_{2}$-module structure and differential we have
\[
\partial_{a=0} (\partial_{b=0}(\lbrace a , b \rbrace + \lbrace -1, ab \rbrace)\cdot x) = \partial_{a=0}(\lbrace a\rbrace \cdot x) = x 
\]
showing that the map $\M^{\bullet}(\bfk) \to \Inv(\Xcal^{\fr}_4, \M)$ given by $x \mapsto \alpha_2(S)\cdot x$ is injective. This concludes the proof of Proposition \ref{prop:explicit}.

\begin{rmk}
Pick the base field $\bfk$ to be a completely real field, so that the elements $\lbrace -1, \ldots, -1 \rbrace$ are all nonzero (because $\H^{\bullet}(\mathbb{R},\ZZ/2\ZZ)\simeq\ZZ/2\ZZ\left[\varepsilon\right]$ where $\varepsilon=\lbrace -1 \rbrace$), and define $F'=\bfk(a,b,c)$.

If we repeat the construction above using $K'=\bfk(\sqrt{a},\sqrt{b},\sqrt{c})$ and picking points points $Q_1, \ldots, Q'_3$ defined over $F'(\sqrt{a}),\ldots, F'(\sqrt{c})$ we obtain
\[\alpha_{\rm tot}(E_{X'})=1 + \lbrace -1, a b c\rbrace + \lbrace a,b \rbrace + \lbrace a, c \rbrace + \lbrace b, c \rbrace + \lbrace -1, -1, -1, abc \rbrace + \]
\[\lbrace -1\rbrace^3 \lbrace a, b, c \rbrace + \lbrace - 1 \rbrace^5 \lbrace abc \rbrace + r\]
where every term in $r$ has degree at least $8$. Note that the classes 
\[1, \alpha_2(E_{X'}), \alpha_4(E_{X'}), \alpha_6(E_{X'})\]
are all $\K_2(\bfk)$-linearly independent, which shows that there must be at least two $\K_2(\bfk)$-linearly independent cohomological invariants of $\Mcal_3$ of degree three or more that do not belong to the submodule generated by invariants of degree two or less.
\end{rmk}

\subsection{Main results}

We are ready to assemble our main results. Combining Corollary \ref{cor:invM3} and Proposition \ref{prop:explicit} we immediately obtain a full computation of the low degree cohomological invariants of $\Mcal_3$. Moreover we can combine Theorem \ref{thm:inv Xdfr} and Proposition \ref{prop:nonHypiso} with Corollary \ref{prop:explicit} to additionally compute the low degree cohomological invariants of $\Mcal_3 \smallsetminus \Hcal_3$ and $\Xcal_4^{\fr}$ over a non algebraically closed field.

\begin{thm}\label{thm:inv M3}
Assume our base field $\bfk$ has characteristic different from $2$, and let $\M$ be an $\ell$-torsion cycle module. We have 
\begin{align*}
    {\rm Inv}^{\leq 2}(\Mcal_3,\M)&=\M^{\leq 2}(\bfk) \oplus \alpha_2 \cdot \M^{0}(\bfk)_2\\
    {\rm Inv}^{\leq 2}(\Mcal_3\smallsetminus \Hcal_3,\M)&=\M^{\leq 2}(\bfk) \oplus \beta_1 \cdot\M^{\leq 1}(\bfk)_9 \oplus \alpha_2 \cdot \M^{0}(\bfk)_2\\
    {\rm Inv}^{\leq 2}(\Xcal_4^{\fr},\M)&=\M^{\leq 2}(\bfk) \oplus \beta_1 \cdot\M^{\leq 1}(\bfk)_9 \oplus \alpha_2 \cdot \M^{0}(\bfk)_2
\end{align*}
where $\alpha_2$ the second Galois-Stiefel-Whitney class of the étale covering $\mathcal{S}^-_3 \to \Mcal_3$ and $\beta_1 \in {\rm Inv}^1(\Mcal_3\smallsetminus \Hcal_3,\K_9)$.
\end{thm}

Restricting the result above to $\M = \H_{\mu_{\ell}^{\vee}}$, and using the fact \cites{EHKV,KV} that for a smooth, generically tame and separated Deligne-Mumford stack whose coarse moduli space is quasi projective the Brauer group and cohomological Brauer group coincide we obtain the computation of the Brauer groups of $\Mcal_3$ and $\Mcal_3 \smallsetminus \Hcal_3$.

\begin{cor}\label{cor:brauer M3}
Assume our base field $\bfk$ has characteristic different from $2$, and let $\ell$ be a positive integer divisible by $2$ and coprime to ${\rm char}(\bfk)$. Then
\[
{\rm Br}(\Mcal_3)_\ell={\rm Br}(\bfk)_\ell \oplus \ZZ/2\ZZ
\]
where the $\ZZ/2\ZZ$ is the pullback of a generator of ${\rm Br}(\Brm{\rm S}_{28})/{\rm Br}(\bfk)$ through the map induced by  $\mathcal{S}^-_3 \to \Mcal_3$. Moreover we have
\[
{\rm Br}(\Mcal_3 \smallsetminus \Hcal_3)_{\ell}={\rm Br}(\Xcal_4^{\fr})_{\ell}={\rm Br}(\Xcal_4)_{\ell}={\rm Br}(\bfk)_\ell \oplus {\rm H}^1(\bfk,\ZZ/9\ZZ)_{\ell} \oplus \ZZ/2\ZZ
\]
where when ${\rm char}(\bfk)$ is coprime to $3$ the ${\rm H}^1(\bfk,\ZZ/9\ZZ)$ component comes from the cup products of $\beta_1 \in {\rm Inv}^1(\Mcal_3\smallsetminus \Hcal_3,\K_9)={\rm H}^1(\Mcal_3 \smallsetminus \Hcal_3,\mu_9)$ and elements of ${\rm H}^1(\bfk,\ZZ/9\ZZ)$.
\end{cor}

We would like to highlight three remaining questions:
\begin{enumerate}
    \item What is the image of ${\rm Br}(\Mcal_3)$ inside ${\rm Br}(\Hcal_3)$?
    \item Is it possible to compute the full cohomological invariants of $\Mcal_3$, at least over an algebraically closed field?
    \item Do the \'etale algebras obtained by odd and even theta characteristics induce any nontrivial invariants on $\Mcal_g$ for $g \geq 4$?
\end{enumerate}

\section{The Brauer group of $\mathcal{A}_3$}

Let $\mathcal{A}_3$ be the moduli stack of three-dimensional principally polarized abelian varieties (p.p.a.v.) over a field $\bfk$ of characteristic different from $2$. We will show that understanding the Brauer groups of $\Mcal_3$ and $\Mcal_3 \smallsetminus \Hcal_3$ is sufficient to compute the Brauer group of $\mathcal{A}_3$ as well.

\begin{rmk}
We already know the Brauer groups of $\mathcal{A}_1$ and $\mathcal{A}_2$. 
\begin{enumerate}
    \item The map $\Mcal_{1,1} \to \mathcal{A}_1$ is an isomorphism, so by \cite{DilPirPositive}*{Thm. 9.1} we know that its Brauer group is
\[
\begin{cases}
{\rm Br}(\mathbb{A}^1_\bfk) \oplus {\rm H}^{1}(\bfk,\ZZ/12\ZZ) & \mbox{if } {\rm char}(\bfk) \neq 2,\\
{\rm Br}(\mathbb{A}^1_\bfk) \oplus {\rm H}^{1}(\bfk,\ZZ/3\ZZ) \oplus {\rm J} & \mbox{if } {\rm char}(\bfk)=2, \, \mathbb{F}_4 \not\subset \bfk, \\
{\rm Br}(\mathbb{A}^1_\bfk) \oplus {\rm H}^{1}(\bfk,\ZZ/12\ZZ) \oplus \ZZ/2\ZZ & \mbox{if } {\rm char}(\bfk)=2, \, \mathbb{F}_4 \subseteq \bfk,
\end{cases}
\]
where ${\rm H}^1(\bfk,\ZZ/4) \subset {\rm J} \subseteq {\rm H}^1(\bfk,\ZZ/8)$ sits in an exact sequence
\[
0 \to  {\rm H}^1(\bfk,\ZZ/4) \to {\rm J} \to \ZZ/2 \to 0. 
\]

\item The Torelli morphism induces an isomorphism $\Mcal_2^{\rm ct}=\overline{\Mcal}_2 \smallsetminus \Delta_0 \simeq \mathcal{A}_2$, where $\Delta_0$ is the divisor parametrizing curves with a non-separating node. This can be seen as the map is finite, representable and birational \cites{LandProper,Seki,MGNoether}. Then by \cite{DilPirRS}*{Prop. 5.2, 5.3} we get that for a base field $\bfk$ of characteristic different from $2$ and an even $\ell$ not divisible by the characteristic
\[
{\rm Br}(\mathcal{A}_2)_{\ell}={\rm Br}(\bfk)_\ell \oplus {\rm H}^1(\bfk,\ZZ/2\ZZ) \oplus \ZZ/2\ZZ.
\]
\end{enumerate}
\end{rmk}

As all the stacks in the section will be smooth, separated and generically tame Deligne Mumford stacks with quasiprojective coarse moduli space, by \cites{EHKV,KV} the $\ell$-torsion parts of their Brauer groups and cohomological Brauer groups coincide. For simplicity we will always refer to the Brauer group even when in practice we are working with the cohomological Brauer group. Recall that the Torelli morphism
\[
\Mcal_3 \xrightarrow{\tau} \mathcal{A}_3, \quad (C \to S) \mapsto ({\rm Jac}(C) \to S, \Theta(C))
\]
\begin{enumerate}
\item induces an open immersion of moduli spaces, whose image is the locus of indecomposable p.p.a.v. (i.e. not isomorphic to a product of lower dimension p.p.a.v.) of dimension three \cites{Col, LandProper};
\item is representable, and the automorphism group of ${\rm Jac}(C)$ is equal to ${\rm Aut}(C)$ if $C$ is hyperelliptic and ${\rm Aut}(C) \times \ZZ/2\ZZ$ if $C$ is not. The extra $\ZZ/2\ZZ$ comes from the canonical involution on a p.p.a.v. that sends an element to its inverse \cite{Seki};
\item is étale of degree two over $\Mcal_3 \smallsetminus \Hcal_3$, and a locally closed immersion when restricted to $\Hcal_3$ \cites{MGNoether,LandHyp}.
\end{enumerate}

We denote by $\mathcal{A}^0_3$ the substack of indecomposable p.p.a.v. of dimension three. The following lemmas are the key to our computation.

\begin{lm}
The $\ell$-torsion part of the Brauer group of $\mathcal{A}_3$ (resp. $\mathcal{A}_3 \smallsetminus \overline{\tau(\Hcal_3)}$) is isomorphic to that of $\mathcal{A}_3^0$ (resp. $\mathcal{A}^0_3 \smallsetminus \tau(\Hcal_3)$).    
\end{lm}
\begin{proof}
The boundary to $\mathcal{A}_3^0$ in $\mathcal{A}_3$ is the substack of p.p.a.v. which can be written as a product. There is an obvious surjective map $\mathcal{A}_2\times \mathcal{A}_1 \cup (\mathcal{A}_1)^3 \to \mathcal{A}_3 \smallsetminus \mathcal{A}^0_3$, showing that it has dimension at most four, and thus it is a closed subset of codimension at least two, so removing it does not modify cohomological invariants and in particular it does not change the $\ell$-torsion of the cohomological Brauer group.    
\end{proof}

\begin{lm}
The map from $\Mcal_3\smallsetminus \Hcal_3$ to the rigidification of $\mathcal{A}^0_3 \smallsetminus \tau(\Hcal_3)$ with respect to the canonical involution is an isomorphism.
\end{lm}
\begin{proof}
 The map in question is representable by schemes exactly as in Prop. \ref{prop:ismomroot} as it is finite and the objects on both sides have the same automorphism groups. It is moreover birational, which implies that it must be an isomorphism by Zariski's main theorem.   
\end{proof}

This is sufficient to show that the Brauer group of $\mathcal{A}_3$ contains that of $\Mcal_3$.

\begin{prop}
The Brauer group of $\mathcal{A}_3^0$ contains a copy of ${\rm Br}(\bfk) \oplus \ZZ/2\ZZ$ which maps bijectively to the corresponding component in the Brauer group of $\Mcal_3$.
\end{prop}
\begin{proof}
The lemma above and our computation of ${\rm Br}(\Mcal_3 \smallsetminus \Hcal_3)$ shows that the Brauer group of $\mathcal{A}^0_3 \smallsetminus \tau(\Hcal_3)$ contains a copy of ${\rm Br}(\bfk) \oplus \ZZ/2\ZZ$ which pulls back to the appropriate elements in the Brauer group of $\Mcal_3 \smallsetminus \Hcal_3$. We just need to show that the element $\alpha$ generating the copy of $\ZZ/2\ZZ$ does not ramify on $\tau(\Hcal_3)$. 

The map $\Mcal_3 \to \mathcal{A}_3$ is flat (being a representable finite map between smooth Deligne-Mumford stacks). Consider the following diagram which is commutative by the compatibility of ramification and flat pullback:
\[
\xymatrixcolsep{1pc}\xymatrix{
 {\rm A}^0(\mathcal{A}^0_3 \smallsetminus \tau(\Hcal_3), \K_2) \ar[d]^{\partial_1} \ar[r]^{\tau^*} & {\rm A}^0(\Mcal_3 \smallsetminus \Hcal_3, \K_2) \ar[d]^{\partial_2}\\
 {\rm A}^0(\tau(\Hcal_3), \K_2) \ar[r]^{\tau^*} & {\rm A}^0(\Hcal_3, \K_2)
 }
\]
the bottom map is the identity as the Torelli map is an immersion when restricted to the hyperelliptic locus. But then we have
\[\tau^*(\partial_1 \alpha) = \partial_2 (\tau^* \alpha) = 0\] 
showing that $\partial_1 \alpha=0$.
\end{proof}

Lastly we need to compute the image of the hyperelliptic locus in the Picard group of $\mathcal{A}_3$.

\begin{lm}
The class of $\tau(\Hcal_3)$ is divisible by $18$ in the Picard group of $\mathcal{A}_3^0$. In particular if ${\rm char}(\bfk) \neq 3$ then $\H^1(\mathcal{A}^0_3\smallsetminus \tau(\Hcal_3), \mu_{18})$ contains an element $\beta$ of order $18$. If ${\rm char}(\bfk)=3$ then we can find an element $\beta$ of order $2$ in $\H^1(\mathcal{A}^0_3\smallsetminus \tau(\Hcal_3), \mu_{2})$.
\end{lm}
\begin{proof}
Let $\Xcal_3 \to \mathcal{A}_3$ be the universal abelian variety and consider the Hodge bundle on $\mathcal{A}_3$. We can see it as the pullback of $\Omega_{\Xcal_3/\mathcal{A}_3}$ through the zero section. Recall that for a smooth projective scheme $\pi\colon X\to S$, the relative tangent sheaf of the Picard variety $\underline{\rm Pic}^0_{X/S}$ is canonically identified with $R^1\pi_*\Ocal_{X}$. Applying this to the case of a smooth curve $C\to S$ of genus three, we deduce that the pullback of the Hodge bundle on $\mathcal{A}_3$ through the Torelli map is equal to the Hodge bundle on $\Mcal_3$.

 Now, in the Picard group of $\Mcal_3$ we have $\left[\Hcal_3\right] = 9\lambda_1 $, where $\lambda_1$ is the first Chern class of the Hodge bundle, and the class of $\tau(\Hcal_3)$ is equal to $\tau_*\left[\Hcal_3\right]$ as the Torelli map is an immersion when restricted to $\Hcal_3$. Using the projection formula we get
\[
\left[\tau(\Hcal_3)\right] = \tau_*(9\lambda_1)=9\tau_*(\tau^*\lambda_1)=18\lambda_1.
\]
We prove the second statement for ${\rm char}(\bfk)\neq 3$: the proof for ${\rm char}(\bfk)= 3$ is identical. Consider the exact sequence
\[
{\rm A}^0(\mathcal{A}_3^0,\K_{18}) \to {\rm A}^0(\mathcal{A}_3^0\smallsetminus \tau(\Hcal_3), \K_{18}) \to {\rm A}^0(\tau(\Hcal_3), \K_{18}) \xrightarrow{i_*} {\rm A}^1(\mathcal{A}_3^0, \K_{18})
\]
By the projection formula the image of $1 \in {\rm A}^0(\tau(\Hcal_3), \K_{18})$ through $i_*$ is zero, so there must be an element $\beta \in {\rm A}^0(\mathcal{A}_3^0\smallsetminus \tau(\Hcal_3), \K^1_{18})$ mapping to it. We conclude by observing that ${\rm A}^0(\mathcal{A}_3^0\smallsetminus \tau(\Hcal_3), \K^1_{18}) = \H^1(\mathcal{A}_3^0\smallsetminus \tau(\Hcal_3), \mu_{18})$ by \cite{DilPirBr}*{Lm. 2.18}.
\end{proof}
We first compute the Brauer group of $\mathcal{A}_3 \smallsetminus \overline{\tau(\Hcal_3)}$.

\begin{prop}\label{prop:BrA3minus}
Let $\bfk$ be a field of characteristic different from $2$, and let $\ell$ be an even positive integer not divisible by ${\rm char}(\bfk)$. Then
\begin{enumerate}
    \item If ${\rm char}(\bfk)=0$ then ${\rm Br}(\mathcal{A}_3 \smallsetminus \overline{\tau(\Hcal_3)}) = {\rm Br}(\bfk) \oplus \H^1(\bfk,\ZZ/18\ZZ) \oplus \ZZ/2\ZZ.$
    \item If ${\rm char}(\bfk)\neq 0$ then ${\rm Br}(\mathcal{A}_3 \smallsetminus \overline{\tau(\Hcal_3)})_\ell = {\rm Br}(\bfk)_\ell \oplus \H^1(\bfk,\ZZ/18\ZZ)_\ell \oplus \ZZ/2\ZZ$.
\end{enumerate}
\end{prop}
\begin{proof}
It suffices to prove the statement for $\mathcal{A}_3^0 \smallsetminus \tau(\Hcal_3)$. 
Consider the $\ZZ/2\ZZ$-gerbe $\mathcal{A}^0_3 \smallsetminus \tau(\Hcal_3) \to \Mcal_3\smallsetminus \Hcal_3$ induced by the rigidification. We know that the induced pullback is injective on cohomology as there is a splitting. The seven-terms exact sequence obtained applying the Leray spectral sequence for $\Gm$ gives us
\[ 0 \to \H^2(\Mcal_3 \smallsetminus \Hcal_3, \Gm) \to {\rm Ker}(\H^2(\mathcal{A}_3^0 \smallsetminus \tau(\Hcal_3),\Gm)\to \H^0(\Mcal_3 \smallsetminus \Hcal_3, {\rm R^2} \pi_* \Gm)) \]\[\to \H^1(\Mcal_3 \smallsetminus \Hcal_3, {\rm R}^1\pi_* \Gm).\]
The stalks of our morphism are equal to $\Brm \ZZ/2\ZZ \times \Spec(R)$ where $R$ is a regular (in fact, smooth over the base) strictly Henselian ring. Using the projection formula we see that 
\[\H^1(\Brm \ZZ/2\ZZ \times \Spec(R), \Gm)={\rm Pic}(\Brm \ZZ/2\ZZ \times \Spec(R))\]
is $2$-torsion. Then we have to check the image of
\[\H^1(\Brm \ZZ/2\ZZ \times \Spec(R), \mu_2)\to \H^1(\Brm \ZZ/2\ZZ \times \Spec(R), \Gm)_2 \to 0.\]
The group on the left is equal to ${\rm Inv}^1(\Brm \ZZ/2\ZZ \times \Spec(R),\K_2)$ which by \cite{DilPirBr}*{Prop. 4.4}  is equal to
\[{\rm Inv}^1(\Spec(R),\K_2) \oplus {\rm Inv}^0(\Spec(R),\K_2)\left[1\right]=\ZZ/2\ZZ\]
because the cohomological invariants with coefficients in $\H_D$ of an Henselian ring that is smooth over the base field are equal to its cohomology (Sec. \ref{sec:coh}, point 3) and $\H^{\bullet}(R,D)=\H^0(R,D)=D$. The copy of $\ZZ/2\ZZ$ cannot map to zero as its pullback to 
\[{\rm Inv}^1(\Brm \ZZ/2\ZZ \times \Spec(R/m),\K_2)=\ZZ/2\ZZ\]
is nonzero. This allows us to conclude that ${\rm R}^1\pi_*\Gm = \ZZ/2\ZZ$. Our exact sequence is now
\[ 0 \to \H^2(\Mcal_3 \smallsetminus \Hcal_3, \Gm) \to {\rm Ker}(\H^2(\mathcal{A}_3^0 \smallsetminus \tau(\Hcal_3),\Gm)\to \H^0(\Mcal_3 \smallsetminus \Hcal_3, {\rm R^2} \pi_* \Gm))\]\[\to \H^1(\Mcal_3 \smallsetminus \Hcal_3, \ZZ/2\ZZ).\]
we are only interested in the $\ell$-torsion part of these groups, so we can restrict to
\[ 0 \to \H^2(\Mcal_3 \smallsetminus \Hcal_3, \Gm)_{\ell} \to {\rm Ker}(\H^2(\mathcal{A}_3^0 \smallsetminus \tau(\Hcal_3),\Gm)\to \H^0(\Mcal_3 \smallsetminus \Hcal_3, {\rm R^2} \pi_* \Gm))_{\ell}\]\[\to \H^1(\Mcal_3 \smallsetminus \Hcal_3, \ZZ/2\ZZ).\]
Now we consider $\H^2(\Brm \ZZ/2\ZZ \times \Spec(R),\Gm)_{\ell}$. We have that
\[\H^2(\Brm \ZZ/2\ZZ \times\Spec(R),\Gm)_\ell={\rm Br}'(\Brm \ZZ/2\ZZ \times \Spec(R))_\ell = {\rm Inv}^2(\Brm \ZZ/2\ZZ \times \Spec(R),\H_{\mu_\ell^\vee})\]
% Using the Hochschild-Serre spectral sequence \cite{Mil}*{Thm 2.20} we have that 
% \[\H^2(\Brm \ZZ/2\ZZ \times \Spec(R),\mu_\ell)=\H^2_{\rm Grp}(\ZZ/2\ZZ,\ZZ/\ell\ZZ)=\ZZ/2\ZZ\]
% which toghether with the fact that ${\rm Pic}(\Brm \ZZ/2\ZZ \times \Spec(R))=\ZZ/2\ZZ$ proves that 
% \[\H^2(\Brm \ZZ/2\ZZ \times \Spec(R),\Gm)_{\ell}=0 \Rightarrow \H^0(\Mcal_3 \smallsetminus \Hcal_3, {\rm R^2} \pi_* \Gm)_\ell=0\]
and again by \cite{DilPirBr}*{Prop. 4.4} this is equal to
\[{\rm Inv}^2(\Spec(R),\H_{\mu_\ell^\vee}) \oplus {\rm Inv}^1(\Spec(R),\H_{\mu_2^\vee})\left[1\right]=0\]
which immediately implies that 
\[\H^0(\Mcal_3 \smallsetminus \Hcal_3, {\rm R^2} \pi_* \Gm)_\ell=0\]
and our sequence becomes
\[ 0 \to \H^2(\Mcal_3 \smallsetminus \Hcal_3, \Gm)_\ell \to \H^2(\mathcal{A}^0_3 \smallsetminus \tau(\Hcal_3),\Gm)_\ell \to \H^1(\Mcal_3 \smallsetminus \Hcal_3, \ZZ/2\ZZ).\]
We know that 
\[\H^2(\Mcal_3 \smallsetminus \Hcal_3, \Gm)_{\ell}={\rm Br}(\bfk)_\ell \oplus {\rm H}^1(\bfk, \ZZ/9\ZZ)_\ell \oplus \ZZ/2\ZZ.\]
Moreover, by our description of the low degree cohomological invariants of $\Mcal_3 \smallsetminus \Hcal_3$ and identifying $\mu_2=\ZZ/2\ZZ$ we get
\[\H^1(\Mcal_3 \smallsetminus \Hcal_3, \ZZ/2\ZZ)={\rm Inv}^1(\Mcal_3 \smallsetminus \Hcal_3,\K_2)=\H^1(\bfk,\ZZ/2\ZZ)\]
so we get
\[0 \to {\rm Br}(\bfk)_\ell \oplus {\rm H}^1(\bfk, \ZZ/9\ZZ)_\ell \oplus \ZZ/2\ZZ \to {\rm Br}(\mathcal{A}^0_3 \smallsetminus \tau(\Hcal_3))_\ell \to \H^1(\bfk,\ZZ/2\ZZ).\]
We will now construct a splitting of the last map.

Consider the $2$-torsion element $9\beta \in \H(\mathcal{A}^0_3 \smallsetminus \tau(\Hcal_3),\mu_2)=\H(\mathcal{A}^0_3 \smallsetminus \tau(\Hcal_3),\ZZ/2\ZZ)$  we defined earlier. Looking at the low degree exact sequence
\[ 0 \to H^1(\Mcal_3 \smallsetminus \Hcal_3, \Gm) \to {\rm H}^1(\mathcal{A}^0_3 \smallsetminus \tau(\Hcal_3), \Gm) \to H^0(\Mcal_3 \smallsetminus \Hcal_3, \ZZ/2\ZZ)=\ZZ/2\ZZ\]
we conclude that the boundary map must send $9\beta$ to $1$. Now if we take the subgroup $9\beta \cap {\rm H}^1(\bfk, \ZZ/2\ZZ) \subset {\rm Br}(\mathcal{A}^0_3 \smallsetminus \tau(\Hcal_3))$, by the compatibility of the spectral sequence with the cap product structure the image of $9\beta \cap x$ in $\H^1(\Mcal_3 \smallsetminus \Hcal_3, \ZZ/2\ZZ)$ is $x$, so that the map $x \mapsto 9\beta \cap x$ is the splitting we were looking for.
\end{proof}

All we have to do now is showing that the extra elements in ${\rm Br}(\mathcal{A}_3 \smallsetminus \tau(\Hcal_3))$ ramify on the image of the hyperelliptic locus.

\begin{thm}
Let $\bfk$ be a field of characteristic different from $2$, and let $\mathcal{A}_3$ be the stack of three-dimensional principally polarized abelian varieties over $\bfk$. Then
\begin{enumerate}
    \item If ${\rm char}(\bfk)=0$ we have ${\rm Br}(\mathcal{A}_3)={\rm Br}(\bfk) \oplus \ZZ/2\ZZ$.
    \item If ${\rm char}(\bfk)=p$ we have ${\rm Br}(\mathcal{A}_3)={\rm Br}(\bfk) \oplus \ZZ/2\ZZ\oplus B''_p$ where $B''_p$ is a $p$-primary torsion subgroup.
\end{enumerate}
\end{thm}
\begin{proof}
It suffices to prove the statement for $\mathcal{A}_3^0$. We already know that the ${\rm Br}(\bfk) \oplus \ZZ/2\ZZ$ component of the Brauer group of $\mathcal{A}^0_e \smallsetminus \tau(\Hcal_3)$ belongs to the Brauer group of $\mathcal{A}_3^0$, so all that is left is to check the $\H^1(\bfk,\ZZ/18\ZZ)$ component. Assume for simplicity that ${\rm char}(\bfk) \neq 3$. Consider the exact sequence 
\[A^0(\mathcal{A}^0_3,\H_{\mu_{18}^{\vee}}) \to A^{0}(\mathcal{A}^0_3 \smallsetminus \tau(\Hcal_3),\H_{\mu_{18}^{\vee}}) \to A^{0}(\tau(\Hcal_3),\H_{\mu_{18}^{\vee}}).\]

By construction $\beta \in A^{0}(\mathcal{A}^0_3,\K^1_{18})$ maps to $1$ in $A^{0}(\tau(\Hcal_3),\K^0_{18})$, so by the compatibility of the boundary with the product \cite{Rost}*{R3f, p.330} we see that the subgroup $\beta \cdot \H^1(\bfk,\ZZ/18\ZZ) \subset A^0(\mathcal{A}^0_3,\H^2_{\mu_{18}^{\vee}})$ maps injectively to $\H^1(\bfk,\ZZ/18\ZZ) \subset A^0(\tau(\Hcal_3),\H^1_{\mu_{18}^{\vee}})$, proving that it does not belong to the Brauer group of $\mathcal{A}^0_3$.

If ${\rm char}(\bfk)=3$ we can just repeat the same reasoning using $2$ and $9\beta$ instead of $18$ and $\beta$.
\end{proof}

\end{document}